\documentclass[a4paper,10pt,reqno]{amsart}

\usepackage[UKenglish]{babel}
\usepackage{amsmath}
\usepackage{amssymb}
\usepackage{amsthm} 
\usepackage{mathtools}
\usepackage{verbatim,comment}
\usepackage{amssymb}
\usepackage[header,title]{appendix}
\usepackage{cite}	
\usepackage{dsfont}

\numberwithin{figure}{section}

\usepackage[backref=page]{hyperref} 
\hypersetup{
    colorlinks=true,
    linkcolor=blue,
    citecolor=red,
    filecolor=magenta,      
    urlcolor=cyan,
    linktocpage=true,
}

\usepackage[utf8]{inputenc}

\usepackage{bbm}
\usepackage{verbatim, stackrel}

\usepackage[shortlabels]{enumitem}
\usepackage{marginnote}

\usepackage{tikz}
\usepackage{tikz-cd}

\newcommand{\bbC}{\mathbb{C}}

\newcommand{\bbH}{\mathbb{H}}

\newcommand{\bbN}{\mathbb{N}}

\newcommand{\bbP}{\mathbb{P}}

\newcommand{\bbR}{\mathbb{R}}

\newcommand{\bbZ}{\mathbb{Z}}

\DeclareMathOperator{\re}{Re} 
\DeclareMathOperator{\dist}{dist} 
\newcommand{\argument}{\mathord{\,\cdot\,}} 
\newcommand{\dx}{\;\mathrm{d}} 
\newcommand{\modulus}[1]{\left\lvert #1 \right\rvert} 
\newcommand{\duality}[2]{\left\langle#1\, ,\, #2\right\rangle} 
\newcommand{\dom}[1]{\operatorname{dom}\left(#1\right)} 
\DeclareMathOperator{\Ima}{Rg} 
\newcommand\restrict[1]{|_{#1}}

\newcommand{\tn}{\textnormal}
\newcommand{\inv}{^{-1}}
\newcommand{\dO}{\partial\Omega}
\newcommand{\R}{\bbR}
\newcommand{\C}{\bbC}
\newcommand{\N}{\bbN}

\newcommand{\Z}{\bbZ}
\newcommand{\di}{{\operatorname{div}}}  

\newcommand{\gr}{\nabla}
\newcommand{\Ae}{A_\tn{ext}}
\newcommand{\Ato}{A_{21}}
\newcommand{\sca}[2] {\left(#1\, ,\, #2\right)} 
\newcommand{\io}{\int_\Omega}
\newcommand{\iow}{\int_{\Omega_w}}
\newcommand{\ioh}{\int_{\Omega_h}}
\newcommand{\ig}{\int_\gamma}
\newcommand{\iG}{\int_\Gamma}
\newcommand{\supp}{\operatorname{supp}}
\newcommand{\tr}{{\operatorname{tr}}} 
\newcommand{\trGw}{\tr^{\Gamma_w}}
\newcommand{\trGh}{\tr^{\Gamma_h}}
\newcommand{\trg}{\tr^{\gamma}}
\newcommand{\trG}{\tr^{\Gamma}}
\newcommand{\hneh}{H^{-1/2}}
\newcommand{\hnehz}{H^{-1/2}_{0}}
\newcommand{\heh}{H^{1/2}}
\newcommand{\hehz}{H^{1/2}_{0}}

\newcommand{\bb}{\gg}

\newcommand{\TS}{T_S^\perp} 
\newcommand{\AS}{A_S^\perp} 
 
\DeclareMathOperator*{\essinf}{ess\,inf}
\newcommand{\domS}{L^\infty(\Omega_h,\C^{n\times n})} 

\makeatletter
\newcommand{\sectionnotoc}[1]{%
  \begingroup
  \let\@tocwrite\@gobbletwo
  \section*{#1}%
  \endgroup
}
\makeatother

\theoremstyle{definition}
\newtheorem{definition}{Definition}[section]
\newtheorem{remark}[definition]{Remark}

\newtheorem*{remark*}{Remark}
\newtheorem*{remarks*}{Remarks}

\theoremstyle{plain}
\newtheorem{proposition}[definition]{Proposition}
\newtheorem{lemma}[definition]{Lemma}
\newtheorem{theorem}[definition]{Theorem}
\newtheorem{corollary}[definition]{Corollary}

\numberwithin{equation}{section} 

\begin{document}

\title[Decay Rates of Coupled Wave-Heat Systems]{Decay Rates and Domain Dependence of a Coupled Wave-Heat System with Spatially Dependent Heat Coefficients}

\author{Pascal Heymoß}

\address[P. Heymoß]{University of Wuppertal, School of Mathematics and Natural Sciences, Gaußstraße 20, 42119 Wuppertal, Germany}
\email{heymoss@uni-wuppertal.de}
\subjclass[2010]{}
\keywords{Coupled wave-heat systems, non-uniform stability, decay rates, closure relations}
\date{\today}
\begin{abstract}
   We study of the long-term behavior of a coupled wave-heat system.   
The system consists of a wave equation and a heat equation on two adjacent Lipschitz domains coupled by a common interface, with the heat equation being allowed to incorporate spatially dependent coefficients. 
    
We first establish strong asymptotic stability independent of the domains and coefficients. 
To this end, we employ the framework of closure relations, which reduces the spectral analysis of the coupled system to that of a wave equation.
    
Secondly, we analyze non-uniform decay rates for classical solutions to the coupled system. Using a non-orthogonal decomposition of the state space, we reduce the problem to a resolvent-type estimate which is independent of the heat domain and heat coefficients.
With this, we extend the known non-uniform decay rates to spatially dependent heat coefficients, yielding logarithmic decay under no assumptions and polynomial decay under the Geometric Control Condition.
\end{abstract}

\maketitle

\vspace{-9mm}
\tableofcontents
\section{Introduction}
\label{sec:mainintro}

We consider a bounded Lipschitz domain $\Omega\subset\R^n$ with two subsets $\Omega_w,\Omega_h$, which are themselves disjoint non-empty Lipschitz domains with $\overline{\Omega} = \overline{\Omega_w\cup \Omega_h}.$ We denote by $\nu_w,\nu_h$ the respective unit outer normal vectors of $\Omega_w,\Omega_h$,  by $\Gamma_w \coloneq \dO\backslash \dO_h,\Gamma_h \coloneq \dO\backslash \dO_w$ the outer parts of the boundary, and by $\gamma \coloneq \dO_w\backslash\overline{\Gamma_w} = \dO_h\backslash \overline{\Gamma_h}$ the interface.    
The triple $(\Omega,\Omega_w,\Omega_h)$ is then called a \textit{Lipschitz interface triple.} See also Figure~\ref{fig:interfacetriple} and  Definition~\ref{def:lipschitz}.

We study the long-term behavior of the  coupled wave-heat system 
\begin{alignat}{2}
    w_{tt}(t, x) &= \Delta w(t, x)~&& \tn{for }x\in \Omega_w, t\geq 0, \nonumber \\ 
    h_t(t, x) &= \di (S\nabla h(t,\cdot))(x)~&&\tn{for }x\in \Omega_h, t\geq 0, \nonumber\\
    w_t(t, x) &= 0 ~&&\tn{for }x\in \Gamma_w, t\geq 0,\nonumber\\
    h(t, x) &= 0~&&\tn{for }x\in \Gamma_h, t\geq 0, \label{eq:1}\\
    w_t(t, x) &= h (t, x) ~&&\tn{for } x\in \gamma, t\geq 0, \nonumber\\
    -\nabla w(t, x)\cdot\nu_w(x) &= (S(x)\nabla h (t, x))\cdot\nu_h (x)~&&\tn{for }x\in \gamma, t\geq 0, \nonumber 
\end{alignat} 
where the heat coefficient $S\in \domS$ is strictly accretive, i.e., where $S$ satisfies $\re(\overline{v}^\top S(x)v)\geq \kappa |v|^2$ for some $\kappa>0$, all $v\in \bbC^n$ and almost all $x\in \Omega_h$.

\vspace{-1mm}
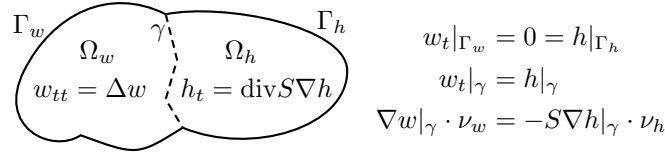
\begin{figure}[htbp]
    \centering
\begin{tikzpicture}[thick, scale=1]
    \begin{scope}[shift={(0,0)}]
        \coordinate (A) at (0, 1.5);
        \coordinate (B) at (0.15, 0.9);
        \coordinate (C) at (0, 0.375);
        \coordinate (D) at (0.225, 0);

        \draw (A) .. controls (-0.75, 1.875) and (-1.5, 1.5) .. (-1.875, 0.75) 
                 .. controls (-2.25, 0) and (-1.5, -0.35) .. (-1.125, -0.1)
                 .. controls (-0.375, -0.375) .. (D);
        
        \draw[densely dashed] (A)node[below] {\hspace{-2mm}$\gamma$} -- (B) -- (C) -- (D) ;

        \draw (A) .. controls (0.75, 1.65) and (2.4, 1.35) .. (2.4, 0.75)
                        .. controls (2.4, 0) and (1.125, -0.375) .. (D);

        \node at (-0.9, 1) {$\Omega_w$};
        \node at (1, 1) {$\Omega_h$};
        \node at (-1, 0.5) {$w_{tt}=\Delta w$};
        \node at (1.2, 0.5) {$h_t = \di S \gr h$};
        \node at (-1.8, 1.35) {$\Gamma_w$};
        \node at (2.2, 1.4) {$\Gamma_h$};
      
        \node  at (4.7,0.6){
            $\begin{aligned}
                w_t|_{\Gamma_w} &= 0 = h|_{\Gamma_h} \\
                w_t|_{\gamma} &= h|_{\gamma}\\
                \gr  w|_{\gamma}\cdot \nu_w &= - S \gr h\restrict{\gamma}\cdot \nu_h
            \end{aligned}$
          };
    \end{scope}

\end{tikzpicture}

\caption{\textbf{Lipschitz interface triple.} An example of a Lipschitz interface triple in $\R^2$, on which a coupled wave-heat system is considered. We note that it is possible for one of the domains $\Omega_w,\Omega_h$ to completely surround by the other.}
    \label{fig:interfacetriple}
\end{figure}

\subsection*{Existing Literature}
The coupled wave-heat system~\eqref{eq:1} was first introduced with constant heat coefficients $S\equiv 1$ by Zhang and Zuazua in \cite{zuazua1d} as a linearized model for fluid-structure interaction.
Such models describe the interaction of elastic solids with fluids, where the fluid-structure coupling matches the normal surface velocity of the two interacting substances, see \cite{zuazualongtime} as well as the references therein. 
System~\eqref{eq:1} is often considered with the simpler transmission condition
$w(\cdot ,t)\restrict{\gamma} = h(\cdot, t)\restrict{\gamma}$ instead of transmission condition $w_t(\cdot ,t)\restrict{\gamma} = h(\cdot, t)\restrict{\gamma}$\cite{zuazua1dsimilar,zuazuapolynomial,DuyckaertsOptimal}.
The transmission condition we consider here is more natural in the sense that both 
$w_t$ and $h$ represent velocities. Indeed, $h$ represents the velocity of the fluid and $w$ represents the displacement of the elastic solid.

Fluid-structure interaction models have been of long-standing interest in the mathematical sciences with applications ranging from engineering (airflow along an aircraft, ocean engineering) to biology (deformation of heart valves) \cite{bookonfluidstructure,aerospaceappl,Oceanengineering,heartvalves}.
For a survey on more detailed models of fluid-structure interaction, see \cite{Avalossurvey}. 

Coupled wave-heat systems, and their asymptotic behavior especially, have received considerable mathematical attention recently, under multiple different geometric assumptions on the underlying domains 
\cite{1dimoptimalrate,rectangulardomains,sandwichdomains,infiniteheatpart}.
In \cite{zuazua1d,zuazua1dsimilar}, controllability and observability of such systems is discussed. 

Furthermore, several studies explore closely related variants of the wave-heat system~\eqref{eq:1}. 
In \cite{lassinontrivialwaveparameter}, a one-dimensional wave-heat system with non-constant density $\rho(x)$ and Young modulus $T(x)$ is considered, where the wave equation takes the form 
$$
    \rho(x)w_{tt}(t,x) = (T(x)w_x(t,x))_x,
$$ 
with Dirichlet and acoustic boundary conditions.
In \cite{colemangurtinlaw,colemangurtinoriginal}, the heat equation on $\Omega_h$ is replaced by a Coleman-Gurtin equation
$$h_t(t,x) =  h_{xx}(t,x) +\int_0^\infty g(s)h_{xx}(t-s,x)\dx s,$$ where $g$ is a heat conductivity kernel, which weights the influence of past temperature on the present state.
In \cite{avaloskleingordonoriginal,avaloskleingordonoptimal}, the wave equation is changed  to a Klein-Gordon equation
$$w_{tt}(t,x) = \Delta w(t,x)-w(t,x).$$
In \cite{degenerateSfourintervals,degenerateSoptimal,degenerateSoriginal}, degenerate  heat coefficients of the form
$$S(x)\approx \dist(x,\gamma)^\alpha,~~\alpha \in [0,1)$$ are considered.
In addition to \cite{degenerateSfourintervals,degenerateSoptimal,degenerateSoriginal}, we refer to the introduction of \cite{networkstarshaped} for an overview of how the stability of the resulting wave-heat system depends on $\alpha$, as well as on how such heat coefficients describe physical phenomena.

Moreover, in \cite{networksoptimalrates,networksplanar,networkstarshaped}, the authors investigate coupled wave-heat networks on graphs, where each edge hosts either a heat or wave equation, and the coupling enforces matching boundary traces at adjacent vertices.

For a stability analysis on abstract coupled systems we refer to \cite{lassiabstractcoupledsystems,lassinontrivialwaveparameter}.

\subsection*{Contributions}

Let us first state one of our main results in Theorem~\ref{thm:mildsolutions} before 
we describe all contributions of the paper along with its structure. 
The wave-heat system~\eqref{eq:1} can be written as an abstract Cauchy problem 
\begin{align*}
    \partial_t
    \begin{pmatrix}
        w_t\\\gr w\\h
    \end{pmatrix} 
    &= A_S
    \begin{pmatrix}
        w_t\\\gr w\\h
    \end{pmatrix}
     \coloneq 
    \begin{pmatrix}
       0 & \di &0\\
       \gr & 0& 0\\
       0& 0&\di S\gr 
    \end{pmatrix}
    \begin{pmatrix}
        w_t\\\gr w\\h
    \end{pmatrix}
\end{align*} 
with 
\begin{align*}
    \dom{A_S} 
    &= 
    \left\{
        \begin{pmatrix}
            w_1\\w_2\\h_1
        \end{pmatrix}
        \in H^1_{\Gamma_w}(\Omega_w)\times H^\di   (\Omega_w)\times H^1_{\Gamma_h}(\Omega_h)~ \right|
        ~S\gr h_1 \in H^\di  (\Omega_h),  \\
        &\hspace{6ex}  \trg(w_1)=\trg(h_1), \; \trg_N(w_2)=-\trg_N(S \gr h_1)\left. \rule{0cm}{0.75cm}
    \right\}. 
\end{align*}  
For the definitions of the above Sobolev spaces and trace operators, see Definitions~\ref{def:ofh1/2},~\ref{def:restrictedtraceoperators}, and~\ref{def:allHspaces}.

We now formulate one of our main results which  shows convergence of mild solutions to $\eqref{eq:1}$ for $t\to \infty$. The other main results are given in Theorem~\ref{thm:domdependence}, Theorem~\ref{thm:logdecay}, and Corollary~\ref{cor:GGCpolynomrate}, see also Figure~\ref{fig:implications}.

\begin{theorem} \label{thm:mildsolutions}
    For each strictly accretive $S\in \domS $,
    the operator $A_S$ generates a contractive $C_0$-semigroup $T_S = (T_S(t))_{t \ge 0}$ on $ L^2(\Omega_w)\times L^2(\Omega_w)^n \times L^2(\Omega_h)$ with fixed space 
    $$
        \ker(A_S) 
        = 
        \left\{
            \begin{pmatrix}
                0\\w_2\\0
            \end{pmatrix}
            ~\middle|~ 
            w_2 \in H^\di(\Omega_w), \; \di (w_2) = 0, \; \trg_N(w_2)=0
        \right\}
    $$ 
    and $T_S(t)$ converges strongly as $t \to \infty$. 
    Moreover, 
    \begin{align*}
        \ker(A_S)^\perp = L^2(\Omega_w)\times \gr H^1_{\Gamma_w}(\Omega_w)\times L^2(\Omega_h)
    \end{align*} 
    is invariant under $T_S$, and its restriction $\TS \coloneq T_S\restrict{\ker(A_S)^\perp}$ is generated by $\AS  \coloneq  A_S\restrict{\ker(A_S)^\perp}$ and converges strongly to $0$ as $t \to \infty$.
\end{theorem}

The proof of Theorem~\ref{thm:mildsolutions} uses the closure relation framework developed in \cite{hansOGpaper,gorrec,gluck2024stability}.
The framework uses methods from the theory of port-Hamiltonian systems, and allows us to reduce the generator $A_S$ to two simpler operators $\Ae,D_D$ acting on  larger domains, which are defined in Subsection~\ref{sec:contractionsemigroup}.  
The proof of Theorem~\ref{thm:mildsolutions} then follows in Subsection~\ref{sec:strongstability}

For this, we require some information about whether two functions can be glued together without loss of their regularity. 
Such results are also known as patching results and are presented in Proposition~\ref{prop:patching}, their proof is postponed to Appendix \ref{sec:appendix}. To this end, and to formalize the boundary conditions in system~\eqref{eq:1}, we discuss different types of trace operators and trace spaces in Section~\ref{sec:maintraces}.

In Section~\ref{sec:mainnonuniform}, we investigate non-uniform decay rates on the semigroup $\TS$, i.e., stability results of the form 
\begin{align*}
    \|\TS(t)x\|_{2} \leq C_k d(t)^k \|x\|_{(\AS)^k} ~~ \forall t\geq 0, \;  x\in \dom{(\AS)^k}, \; k \in \N,
\end{align*} 
where $d:[0,\infty)\to[0,\infty)$ is a decreasing function with $\lim_{t\to \infty} d(t) = 0$, the $C_k$ are positive constants, and $\|\argument\|_A$ denotes the graph norm of an operator $A$. 
Such non-uniform decay rates are stability results for classical solutions to the wave-heat system \eqref{eq:1}.
 
Using some of the resolvent criteria  for non-uniform decay rates established in \cite{borichev2010optimal,batty2008non}, which we recall in Theorem \ref{thm:abstractresolventbounds}, we 
infer equivalent characterizations of such non-uniform decay, allowing us to further analyze how the underlying domains $\Omega_w,\Omega_h$ and the interface $\gamma$ affects these. 
Specifically, we show that such non-uniform decay is (up to some small loss in the decay rate $d$) independent of $\Omega_h$ and $S$, instead depending only on $\Omega_w$ and the interface $\gamma$. For this, see Theorem~\ref{thm:domdependence}, where a monotonicity property with respect to $(\Omega_w,\gamma)$ is also established.

In Section~\ref{sec:mainspecificrates}, we demonstrate that $\TS$ has
logarithmic decay rates ($d(t) = 1/\log(t)$) for each strictly accretive $S \in \domS$ on every Lipschitz interface triple $(\Omega,\Omega_w,\Omega_h)$, see Theorem~\ref{thm:logdecay}.
This logarithmic decay result is strictly stronger than the strong stability result in Theorem \ref{thm:mildsolutions}, as it provides
even a lower bound on the rate of decay.

We then consider the case where $n=1$ or where $\gamma$ geometrically controls $\Omega_w$, see Definition~\ref{def:GCC}, in which we infer that $\TS$ has
polynomially fast decay rates ($d(t) = t^{-1/2}$) for each strictly  accretive $S\in \domS$, see Corollary~\ref{cor:1dpolynomialdecay} and Corollary~\ref{cor:GGCpolynomrate}.
We note that the semigroup $\TS$ does not converge exponentially fast to zero in operator norm. This was shown in Chapter~6 of \cite{zuazualongtime} via a construction of solutions which are highly localized around the trajectory of reflective rays inside the non-dissipative wave domain $\Omega_w$.

All our methods, as well as the results of Section~\ref{sec:mainnonuniform}, are new, while the results of Section~\ref{sec:mainstability} and Section~\ref{sec:mainspecificrates} were previously  known in the special case $S\equiv 1$.
Specifically under the assumption $S\equiv 1$, the strong stability Theorem~\ref{thm:mildsolutions} was first shown by Zhang and Zuazua in Theorem 4 of \cite{zuazualongtime}, the logarithmic decay was shown by Fathallah \cite{ines}, the polynomial decay under the assumption that $\gamma$ geometrically controls $\Omega_w$ was shown by Zhang and Zuazua in Chapter~7 of \cite{zuazualongtime} and later improved by Duyckaerts \cite{DuyckaertsOptimal}, and polynomial decay in one dimensions with rate $d(t) = t^{-2}$ was first established by Zhang and Zuazua \cite{zuazua1d} and then proven to be optimal by Batty, Paunonen, and Seifert \cite{1dimoptimalrate}.

A large part of this paper is thus focused on  extending the existing results to more general settings with weaker assumptions, most importantly allowing for non-constant heat coefficients $S\in \domS$. In particular, we impose no smoothness on $S.$
We also note that the regularity assumptions imposed on $(\Omega,\Omega_w,\Omega_h)$ in \cite{ines,zuazualongtime} are stronger than our assumptions here, as we only require Lipschitz regularity on the domains and the interface, see Definition~\ref{def:lipschitz}.

A major difficulty if $S \not= 1$ is that $\dom{\di S\gr}$, and hence $\dom{A_S}$, depend intricately on $S$. We circumvent this difficulty by decomposing $\dom{\AS}$ non-orthogonally into a part that is independent of $\Omega_h$, and hence independent of $S$, and a part whose $L^2$-mass is supported mostly on $\Omega_h$, where strong dissipation occurs. This will be the main focus of Subsection~\ref{sec:reductiontowavedomain}, allowing us to extend the existing results in \cite{zuazualongtime,zuazua1d,ines} to non-trivial, strictly accretive heat coefficients $S\in \domS$.

\section{Trace Preliminaries} \label{sec:maintraces}
As mentioned in the introduction, this section is dedicated to the different types of trace operators and Sobolev spaces needed for the analysis of the wave-heat system~\eqref{eq:1}. Additionally, these traces are essential in determining when two functions defined on opposite parts of a Lipschitz interface triple can be glued together without loss of their regularity, see Proposition~\ref{prop:patching}.
First, however, we present the precise regularity assumptions on Lipschitz interface triples.

\begin{definition}[Lipschitz interface triples]\label{def:lipschitz}
\begin{enumerate}
    \item Let $\Omega\subset \R^n$ be a non-empty domain. We call $\Omega$ a (strongly) \textit{Lipschitz domain} if $n=1$ or if for every $x\in \dO$, there is an open neighborhood $U_x \subset \R^n$ around $x$, an orthogonal matrix $O_x \in \R^{n\times n}$ and a Lipschitz function 
    $\phi_x: \R^{n-1}\to \R$ with 
    \begin{align*}
       O_x (U_x\cap \Omega) = \{(\xi,y) \in \R^{n-1}\times \R~|~ y < \phi_x(\xi) \}\cap O_x(U_x).
    \end{align*}

    \item Let $\Omega\subset \R^n$ be a non-empty, bounded Lipschitz domain, and let $\Omega_w,\Omega_h\subset \Omega$ be two disjoint, non-empty Lipschitz domains. If 
    $\overline{\Omega}=\overline{\Omega_w\cup \Omega_h}$, we call $(\Omega,\Omega_w,\Omega_h)$ a \textit{Lipschitz interface triple} with interface $\gamma.$ 
    In this case, we set
    $$\Gamma_w  \coloneq  \dO\backslash\dO_h, ~~\Gamma_h \coloneq \dO\backslash\dO_w,~~ \gamma \coloneq  \dO_w\backslash\overline{\Gamma_h} = \dO_h\backslash\overline{\Gamma_w}.$$
    Furthermore, we denote by $\nu_w,\nu_h$ the unit outer normal vectors on $\Gamma_w,\Gamma_h$, which exist almost everywhere due to Rademacher's Theorem for $n\geq 2,$ and trivially everywhere in the case $n=1.$

Our assumptions on $(\Omega,\Omega_w,\Omega_h)$ are weaker than the assumptions imposed in  \cite{ines,zuazualongtime}.

    \item Let $\Omega\subset \R^n$ be a non-empty Lipschitz domain, and let $\Gamma\subset \dO$ be a non-empty, relatively open subset of the boundary.
    We call $(\Omega,\Gamma)$ an \textit{admissible} Lipschitz pair, if there exists some domain $\Tilde{\Omega}\subset \R^n$ such that 
    $$(\Omega\cup \Gamma\cup \Tilde{\Omega}, \Omega, \Tilde{\Omega})$$ is a Lipschitz interface triple with interface $\Gamma$.
\end{enumerate}
\end{definition}

\begin{definition} \label{def:ofh1/2}
    Let $\Gamma$ be a relatively open subset of the boundary $\partial \Omega$ of some bounded Lipschitz domain $\Omega \subset \R^n, n\geq 2$.
    We then define the space
    \begin{alignat}{2}
        H^{1/2}(\Gamma)& \coloneq  \left\{v\in L^2(\Gamma)~\middle|~ \|v\|_{H^{1/2}(\Gamma)}^2 <\infty \right\},  \tn{ where }\nonumber\\
        \|v\|_{H^{1/2}(\Gamma)}^2& \coloneq  \|v\|_{L^2(\Gamma)}^2\nonumber
        + \iG \iG \frac{|u(x)-u(y)|^2}{|x-y|^{n}}\dx\sigma(x,y).
    \end{alignat}   
    If $n=1$, we equip $\heh(\Gamma)  \coloneq  \C^{|\Gamma|}$ with the Euclidean norm.
    
    Additionally, we denote by $$\tr: H^1(\Omega) \to H^{1/2}(\partial \Omega)$$ the usual trace operator extending the restriction $C^1(\overline{\Omega})\to C^0(\partial \Omega), v\mapsto v\restrict{\partial\Omega}$, which is bounded and surjective, see Theorem 1.5.1.3 in \cite{grisvard2011elliptic}.

\end{definition}

\begin{definition}
Let $\Omega\subset \R^n$ be a bounded Lipschitz domain with unit outer normal $\nu.$
    We denote by $\tr_N$ the normal trace operator
  \begin{align*}
      \tr_N&:H^\di(\Omega) \to H^{-1/2}(\partial \Omega) = (H^{1/2}(\partial\Omega))',
  \end{align*}    
    which extends the normal component trace operator 
    $$H^1(\Omega)\to H^{1/2}(\dO), f\mapsto \tr (f) \cdot \nu.$$
    The normal trace operator $\tr_N$ is well-defined, bounded and surjective, see e.g. Lemma 20.2 in \cite{tartar2007introduction}.
 For $f\in C^1(\Omega)^n, \psi \in H^1(\Omega)$,  we note the identity 
 \begin{align} 
    \int_{\partial \Omega}  \tr(\psi) f\cdot \nu \dx \sigma &= \io \di(f\psi)\dx x \label{eq:motivationnormaltrace}
    \\&= \io f\cdot \gr \psi + \di(f)\psi \dx x,\nonumber
 \end{align} which extends by density to 
\begin{align}
      \duality{\tr_N(f)}{\tr (\psi)} &= \io f\cdot \gr \psi + \di(f) \psi \dx x \label{eq:normaltraceidentitybydensity}
  \end{align} for $f\in H^\di( \Omega),\psi \in H^{1}( \Omega).$
      Here and in the following, we understand that if $\Omega$ is a one-dimensional domain and $f$ is some function defined on $\dO$, then $$\int_{\dO} f\cdot \nu\dx\sigma = f\left(\sup_{x\in \Omega} x\right)-f\left(\inf_{x\in \Omega} x\right) ~\tn{ and }~\iG f\cdot \nu \dx\sigma = \int_{\dO} \mathds{1}_\Gamma f\cdot \nu \dx\sigma. $$
 \end{definition}
 So far, these traces are only  defined on the whole boundary of some bounded Lipschitz domain $\Omega\subset \R^n$. However, the boundary conditions in \eqref{eq:1} are restricted to the interface $\gamma$ or the outer boundary parts $\Gamma_w,\Gamma_h$ of the Lipschitz interface triple $(\Omega,\Omega_w,\Omega_h)$. We thus need to work with trace operators on  relatively open subsets $\Gamma$ of the boundary $\partial \Omega$.
 One of the main subtleties is that extensions of $H^{1/2}(\Gamma)$ functions by zero generally do not lie in $H^{1/2}(\dO)$. 
 
 \begin{definition}
 Let $\Gamma$ be a relatively open subset of the boundary $\partial \Omega$ of some bounded Lipschitz domain $\Omega \subset \R^n$.
 \label{def:restrictedtraceoperators}
 \begin{enumerate}
     \item  The restriction of $\tr(f)$ to $\Gamma$ for $f\in H^1(\Omega)$ induces the restricted trace operator 
    $$\trG:H^1(\Omega)\to H^{1/2}(\Gamma), f\mapsto \tr(f)\restrict{\Gamma},$$
    which is again bounded and surjective, see e.g.  Proposition 1.2.60 in \cite{Meinlschmidt}.
    \item We define the Lions-Magenes space $\hehz(\Gamma)$ as the subspace of functions $f\in \heh(\Gamma)$, for which the extension by zero $\Tilde{f}$ to $\dO$ is again an element of $\heh(\dO)$. If $n=1$, one clearly has $\hehz(\Gamma) = \heh(\Gamma)$. If $n\geq 2$, then $\hehz(\Gamma)$ is an intrinsic, i.e., independent of $\dO$, Hilbert space when equipped with the norm
    \begin{align*}
        \|f\|_{\hehz(\Gamma)}  \coloneq  \left(\|f\|_{\heh(\Gamma)} +\iG  \frac{|u(x)|^2}{\dist(x,\partial \Gamma)} \dx x\right)^{1/2},
    \end{align*} see e.g.\ Chapter~33 in \cite{tartar2007introduction}.
    This norm is equivalent to the $H^{1/2}(\dO)$-norm of the zero extension $\Tilde{f}$.
    Additionally, denote by $\hnehz(\Gamma)$ the dual space of $\hehz(\Gamma)$.
    \item We define the restricted normal trace operator
    $$\trG_N:H^\di(\Omega)\to \hnehz(\Gamma), f\mapsto \left( \hehz(\Gamma)\to \C, g\mapsto
    \duality{\tr_N(f)}{\Tilde{g}}\right),$$ where 
    $\Tilde{g}$ denotes the extension by zero to $\dO$ of an element $g\in \hehz(\Gamma)$.
    
    The restricted normal trace operator $\trG_N:H^\di(\Omega)\to \hnehz(\Gamma)$ is well-defined, bounded and surjective, see Lemma~\ref{lem:normaltracesurj}.
    \end{enumerate}

 \end{definition}
  
\begin{remark}\label{rem:smoothrestrictedtraces}
    \begin{enumerate}
    \item $\trG_N(f)$ is not an intrinsic element of $\hnehz(\Gamma)$, as the orientation of $\Gamma$ determines the direction of the outer normal vector $\nu$. For example, one has $\trG_N(f\restrict{\Omega}) = - \trG_N(f\restrict{\R^n\backslash \Omega})$ for all $f\in H^\di(\R^n).$
    \item As in equation~\eqref{eq:motivationnormaltrace}, we find,  for $f \in C^1(\overline{\Omega})$, that the restricted normal trace is simply given by $\trG_N(f) = (f \cdot \nu)\restrict{\Gamma},$ where $\nu$ denotes the unit outer normal vector on $\dO$.
    If $f$ is only an element of $H^1(\Omega)$, one still has $\trG_N(f) = \trG(f) \cdot \nu.$

    Furthermore, due to the surjectivity of $\trG,$ we find that 
    $\trG_N(f)$ is uniquely determined by the identity
    \begin{align}\label{eq:normaltraceidentity}
        \duality{\trG_N(f)}{\trG(g)} = \io  f\cdot \gr g + \di(f)g \dx  x
    \end{align} for all $g \in H^1(\Omega)$ with $\tr^{\dO \backslash\Gamma}(g) = 0$.
    \end{enumerate}
\end{remark}

\begin{definition}\label{def:allHspaces}

Let $\Gamma$ be a relatively open subset of the boundary $\partial \Omega$ of some bounded Lipschitz domain $\Omega \subset \R^n$. We define the following subspaces of $L^2(\Omega)$.
    \begin{align*}
    H^1_0(\Omega) & \coloneq  \{f\in H^1(\Omega)~|~ \tr(f) = 0\},\\
        H^1_\Gamma(\Omega) & \coloneq  \{f\in H^1(\Omega)~|~ \trG(f) = 0\},\\
        H^\di_0(\Omega)& \coloneq   \{f\in H^\di(\Omega)~|~ \tr_N(f) = 0\},\\
        H^\di_\Gamma(\Omega)& \coloneq   \{f\in H^\di(\Omega)~|~ \trG_N(f) = 0\},\\
        H^\Delta(\Omega)& \coloneq \{f\in H^1(\Omega)~|~ \gr f \in H^\di(\Omega)\},\\
         H^\Delta_0(\Omega)& \coloneq \{f\in H^1_0(\Omega)~|~ \gr f \in H^\di(\Omega)\},\\
          H^\Delta_\Gamma(\Omega)& \coloneq \{f\in H^1_\Gamma(\Omega)~|~ \gr f \in H^\di(\Omega)\}.
    \end{align*}
    The first two spaces are equipped with the usual $H^1$-norm, the third and fourth with the usual $H^\di$-norm,
    and the latter three spaces with the norm $$
        \|f\|_{H^\Delta(\Omega)}^2  \coloneq  \|f\|^2_{H^1(\Omega)}+ \|\Delta f\|^2_{L^2(\Omega)}
    $$
\end{definition}

 The following Proposition~\ref{prop:patching} determines whether two functions on opposite parts of a Lipschitz interface triple can be glued together without losing their regularity, depending on their boundary traces.

\begin{proposition}\label{prop:patching}
Let $(\Omega,\Omega_w,\Omega_h)$ be a Lipschitz interface triple with interface $\gamma$. Let $\pi\in L^2(\Omega)$ be a patching of $w,h$, i.e., $\pi\restrict{\Omega_w} = w, \pi\restrict{\Omega_h }=h$.
    \begin{enumerate}
        \item Assume $w\in H^1(\Omega_w), h\in H^1(\Omega_h)$, then $w,h$ have matching trace $$\trg(w) =\trg(h)$$ if and only if $\pi$ is an element of $ H^1(\Omega)$.
        \item Assume $w\in H^\di(\Omega_w), h\in H^\di(\Omega_h)$, then $w,h$ have matching normal trace $$\trg_N(w) = -\trg_N(h)$$ if and only if $\pi$ is an element of $ H^\di(\Omega)$.
        
        \item Assume $w\in H^\Delta(\Omega_w)$, $h\in H^\Delta(\Omega_h)$, then $w,h$ have matching boundary trace and normal derivative, i.e., $$\trg(w) =\trg(h)~~\tn{and} ~~ \trg_N(\gr w) = -\trg_N(\gr h),$$ if and only if $\pi$ is an element of $ H^\Delta(\Omega)$.
    \end{enumerate}
\end{proposition}
We postpone the proof of this proposition to the beginning of the appendix, see Proposition~\ref{prop:patchinginappendix}.

\section{Wave-Heat Systems and Strong Stability}\label{sec:mainstability}
We are now ready to investigate the long-term behavior of the  coupled wave-heat system~\eqref{eq:1}, which we can be rewritten using the trace operators from Section~\ref{sec:maintraces}:
\begin{alignat}{2}
    w_{tt}(t,\cdot) &= \Delta w(t,\cdot)&&\tn{in }L^2(\Omega_w)\tn{ for }t\geq 0,\nonumber 
    \\ h_t(t,\cdot ) &=  \di (S\nabla h(t,\cdot))~&&\tn{in }L^2(\Omega_h)~\tn{for }t\geq 0,\nonumber \\
    \tr^{\Gamma_w}(w_t(t,\cdot)) &= 0 ~&&\tn{for }t\geq 0,\label{eq:2}\\
    \trGh(h(t,\cdot))&= 0~&&\tn{for }t\geq 0, \nonumber\\
     \trg(w_t(t,\cdot))&= \trg(h(t,\cdot))~&&\tn{for }t\geq 0, \nonumber\\
    -\trg_N(\gr w(t,\cdot)) &= \trg_N(S\gr h(t,\cdot))~&&\tn{for }t\geq 0.\nonumber 
\end{alignat} We recall that the heat coefficient $S\in \domS $ is strictly  accretive and thus satisfies $\re(\overline{v}^\top S(x)v)\geq \kappa |v|^2$ for some number $\kappa>0$, all $v\in \bbC^n$ and almost all $x\in \Omega_h$.

Under these assumptions, we define the space $H^{\di S\gr}(\Omega)$ as
    \begin{align}
         \left\{h\in H^1(\Omega)~\middle|~ \exists g\in L^2(\Omega): \io S\gr h \cdot \gr \phi \dx x= -\io g \phi \dx x ~\forall \phi \in H^1_0(\Omega)\right\} \label{eq:defofdivSgr}
    \end{align} and set $\di S\gr h = g$ for such $h\in H^{\di S\gr}(\Omega)$.
    We equip $H^{\di S\gr}(\Omega)$ with the norm $$\|h\|_{H^{\di S\gr}(\Omega)} \coloneq  \|h\|^2_{H^1(\Omega)} +  \|\di S \gr h \|^2_{L^2(\Omega)},$$ analogously to $H^\Delta(\Omega)$.
If $\Gamma \subset \partial\Omega$ is relatively open, we set $$H^{\di S \gr}_\Gamma(\Omega)  \coloneq  \{f\in H^{\di S \gr}(\Omega)~|~ \trG(f) = 0 \}.$$

\subsection{Closure Relation Framework}\label{sec:contractionsemigroup}

This section is concerned with the proof of Theorem~\ref{thm:mildsolutions}.
First, we construct an auxiliary operator $\Ae$ which generates unitary semigroup on a larger domain. Using the closure relation formalism developed in \cite{gluck2024stability,gorrec,hansOGpaper}, we extend spectral properties of $\Ae$ to the generator $A_S$ (introduced in Theorem~\ref{thm:mildsolutions}) of the wave-heat system, as was done, e.g., in Theorem 2.3 of \cite{gorrec}.

We begin by considering an auxiliary wave equation  on the whole domain $\Omega$.  \begin{alignat}{2}
    w_{tt}(x,t) &= \Delta w(x,t)~&& \tn{for }x\in \Omega,t\geq 0, \label{eq:auxiliarywave}
\\
\tr(w_t(\cdot,t)) &= 0 ~&&\tn{for }t\geq 0,\nonumber
\end{alignat}
The  closure relation formalism presents a method from port-Hamiltonian systems theory, motivating us to write the wave equation in a port-Hamiltonian formulation, i.e., in a first-order formulation.
By setting $x_1 \coloneq  w_t, x_2=\gr w,$ the wave equation can be written as the abstract Cauchy Problem
\begin{align*}
   \begin{pmatrix}
       \dot{x}_1\\\dot{x}_2
   \end{pmatrix} = \begin{pmatrix}
       0 & \di \\ 
       \gr & 0
   \end{pmatrix}\begin{pmatrix}
       x_1\\x_2
   \end{pmatrix}.
\end{align*}
The domain of the block operator
\begin{align}
    D_D \coloneq \begin{pmatrix}
       0 & \di \\ 
       \gr & 0
   \end{pmatrix}:\dom{D_D}\subset L^2(\Omega)\times L^2(\Omega)^n\to L^2(\Omega)\times L^2(\Omega)^n \label{eq:defofDD}
\end{align} is given by 
$\dom{D_D} =  H^1_0(\Omega)\times H^\di(\Omega)$, encoding the Dirichlet boundary conditions into the system.

We restrict the state variables $x_1,x_2$ to $\Omega_w,\Omega_h$ by setting $w_1 \coloneq x_1\restrict{\Omega_w},w_2 \coloneq x_2\restrict{\Omega_w},h_1 \coloneq x_1\restrict{\Omega_h},h_2 \coloneq x_2\restrict{\Omega_h}. $
The wave equation can then be rewritten in these variables as 
\begin{align}
    \begin{pmatrix}
        \dot{w}_1\\\dot{w}_2\\\dot{h}_1\\\dot{h}_2
    \end{pmatrix}=\underbrace{\begin{pmatrix}
        0 & \di & 0 & 0\\
        \gr & 0 & 0 & 0\\
        0& 0 &0 &\di\\0&0&\gr&0
    \end{pmatrix}}_{=:\,\Ae}\begin{pmatrix}
        w_1\\w_2\\h_1\\h_2
    \end{pmatrix}.\label{eq:defofD_D}
\end{align}
Let $\bbH_1 \coloneq L^2(\Omega_w)\times L^2(\Omega_w)^n\times L^2(\Omega_h), \bbH_2 \coloneq  L^2(\Omega_h)^n,$
the operator $\Ae:\dom{\Ae}\subset \bbH_1\times \bbH_2\to \bbH_1\times \bbH_2$ then has domain $$\dom{\Ae} \coloneq \left\{\begin{pmatrix}
       x_1\restrict{\Omega_w}\\x_2\restrict{\Omega_w}\\x_1\restrict{\Omega_h}\\x_2\restrict{\Omega_h}
    \end{pmatrix}\middle|~x_1\in H^1_0(\Omega), x_2\in H^\di(\Omega)\right\}$$ and, by construction, generates an equivalent abstract Cauchy problem as the Dirichlet block operator $D_D$. 
    
In analogy to Definition 2.1 in \cite{gluck2024stability}, we define 
$A_1:\dom{\Ae}\subset \bbH_1\times \bbH_2\to \bbH_1$ as $$A_1  \coloneq   \begin{pmatrix}
    0 & \di & 0 & 0\\
    \gr&0&0&0\\
    0&0&0&\di
\end{pmatrix}, ~~\dom{A_1} \coloneq \dom{\Ae},$$ and $\Ato:\dom{\Ato}\subset \bbH_1\to \bbH_2$ as \begin{align*}
    \Ato &= \begin{pmatrix}
        0&0&\gr
    \end{pmatrix},\\ \dom{\Ato} &= L^2(\Omega_w)\times L^2(\Omega_w)^n\times H^1(\Omega_h),
\end{align*}
such that 
$$\Ae x = \begin{pmatrix}
  \multicolumn{2}{c}{A_1} \\
  \Ato & 0
\end{pmatrix}x ~~\forall x\in \dom{\Ae}.$$
By following the construction in Definition 2.1 (b) of \cite{gluck2024stability}, 
we obtain an operator $A_S:\dom{A_S}\subset \bbH_1\to \bbH_1$ given by \begin{align*} 
    A_S\begin{pmatrix}
        w_1\\w_2\\h_1
    \end{pmatrix}  \coloneq  A_1\begin{pmatrix}
        w_1\\w_2\\h_1\\S \Ato \begin{pmatrix}
            w_1\\w_2\\h_1
        \end{pmatrix}
    \end{pmatrix} = \begin{pmatrix}
        0&\di & 0\\
        \gr &0 &0\\ 
        0 &0&\di S \gr
    \end{pmatrix}\begin{pmatrix}
        w_1\\w_2\\h_1
    \end{pmatrix}
\end{align*}
with \begin{align}
    \dom{A_S}& \coloneq \left\{\begin{pmatrix}
        w_1\\w_2\\h_1
    \end{pmatrix}\in \dom{\Ato}~\middle|\begin{pmatrix}
        w_1\\w_2\\h_1\\S\gr h_1
    \end{pmatrix}\in\dom{\Ae}\right\}
    \label{eq:defofAS}
    \\&= \left\{\begin{pmatrix}
        w_1\\w_2\\h_1
    \end{pmatrix}\middle|~ \begin{pmatrix}
        w_1\\w_2\\h_1\\S\gr h_1
    \end{pmatrix} = \begin{pmatrix}
       x_1\restrict{\Omega_w}\\x_2\restrict{\Omega_w}\\x_1\restrict{\Omega_h}\\x_2\restrict{\Omega_h}
    \end{pmatrix},~x_1\in H^1_0(\Omega), x_2\in H^\di(\Omega)\right\}.\nonumber
\end{align}
This operator is in fact equal to the operator $A_S$ in Theorem~\ref{thm:mildsolutions}. Indeed, using the patching results in Proposition~\ref{prop:patching}, we find 
\begin{align} \dom{A_S} &= \left\{
            \begin{pmatrix}
                w_1\\w_2\\h_1
            \end{pmatrix}
            \in H^1_{\Gamma_w}(\Omega_w)\times H^\di(\Omega_w)\times H^{\di S \gr}_{\Gamma_h}(\Omega_h)~ \right| \nonumber\\&\hspace{6ex}  \trg(w_1)=\trg(h_1), \trg_N(w_2)=-\trg_N(S \gr h_1)\left. \rule{0cm}{0.75cm}\right\}.\label{eq:domAS} 
\end{align}

This leads us to our first intermediate result regarding the semigroup generated by $A_S$, which we denote by $T_S$ as in Theorem~\ref{thm:mildsolutions}.

\begin{lemma}
    \label{lem:contractive} 
    The operator $\Ae:\dom{\Ae}\subset \bbH_1 \times \bbH_2\to \bbH_1\times \bbH_2$ is skew-adjoint and $A_S:\dom{A_S}\subset \bbH_1 \to \bbH_1$  generates a contractive $C_0$-semigroup for each strictly accretive heat coefficient $S\in \domS$.
\end{lemma}
\begin{proof}

To show that $A_S$ generates a contractive $C_0$-semigroup, it suffices to show that   $\Ae$ generates a contractive $C_0$-semigroup on $\bbH_1\times \bbH_2$, see Theorem 2.2 in \cite{gorrec}.
In fact, we show that $\Ae$ is skew-adjoint and generates a unitary semigroup. 

Since $\Ae$ was constructed to generate an equivalent abstract Cauchy problem as $D_D$, it suffices to show that $D_D$ is skew-adjoint on the Hilbert-space  $L^2(\Omega)\times L^2(\Omega)^n$.
This directly follows from the fact that $ - \di:H^\di(\Omega)\to L^2(\Omega)$ is the adjoint operator to $\gr:H^1_0(\Omega)\to L^2(\Omega)^n$ with respect to the $L^2(\Omega)$-inner product.
\end{proof}

\subsection{Strong Stability}\label{sec:strongstability}

As a second step, we deduce  some spectral properties of $A_S$. The closure relation framework allows us to reduce the problem to the simpler, Dirichlet wave equation generators $\Ae$ and $D_D$.
The upcoming proofs closely follow the strategy of Theorem 3.2 and Theorem 3.6 in \cite{gluck2024stability}.

\begin{proposition} \label{prop:kernelandspectrum}
For each strictly accretive heat coefficient $S\in \dom{S}$, the kernel of $A_S$ is independent of $S$ and given by $$\ker(A_S) = \left\{\begin{pmatrix}
        0\\w_2\\0
    \end{pmatrix}~\middle|~ w_2 \in H^\di(\Omega_w), \di (w_2) = 0, \trg_N(w_2)=0\right\}$$ and
    the spectrum of $A_S$ on the imaginary axis is given by $\sigma(A_S)\cap i\R = \{0\}$. 
\end{proposition}
\begin{proof}
      Theorem 3.2 in \cite{gluck2024stability} implies that the kernel of $A_S$ is given by $$ \{v~|~ (v,0)  \in \ker(\Ae)\}.$$ 
    Assume $v = (w_1,w_2,h_1) \in  L^2(\Omega_w)\times L^2(\Omega_w)^n\times L^2(\Omega_h)$ satisfies $A_s v=0$, by the construction of $\Ae$ from $D_D$ we find that there exists an  $x = (x_1,x_2)  \in \dom{D_D} = H^1_0(\Omega)\times H^\di(\Omega)$ with $$x_1\restrict{\Omega_w} = w_1,~x_1\restrict{\Omega_h} = h_1,~x_2\restrict{\Omega_w} = w_2,~x_2\restrict{\Omega_h} = 0$$ and $$\gr x_1\restrict{\Omega_w} = \gr w_1 = 0,~\gr x_1\restrict{\Omega_h} = \gr h_1 = 0,~ \di x_2\restrict{\Omega_w} = \di (w_2) = 0.$$
    Together with $\tr(x_1) = 0$, this implies $x_1 = 0$.
    
    We also note that the patching of $w_2 \in H^\di(\Omega_w)$ and the zero function in $H^\di(\Omega_h)$, which is precisely $x_2$, is an element of $H^\di(\Omega)$. For this, Proposition~\ref{prop:patching}.(ii) necessitates that their normal traces match, i.e., $\trg_N(w_2) = 0.$ This proves the inclusion $$\ker(A_S) \subset \left\{\begin{pmatrix}
        0\\w_2\\0
    \end{pmatrix}~\middle|~ w_2 \in H^\di(\Omega_w), \di (w_2) = 0, \trg_N(w_2)=0\right\}.$$
    The converse inclusion is clear.

  We continue with the proof of the second claim, for which we assume, by contradiction, that there exists an element $i\lambda \in \sigma(A_S)$ with $\lambda \in \R\backslash\{0\}$.
  Following the proof of Theorem 3.6 in \cite{gluck2024stability}, we observe, after rescaling, that there exists an approximate eigenfunction
  $$(w_k)_{k\in \N} = \begin{pmatrix}
      v_k \\ S\Ato v_k
  \end{pmatrix}_{k\in \N}\subset \dom{\Ae} $$ with $$
     \lim_{k\to \infty}\|(i\lambda-\Ae)w_k\|_{\bbH^1\times \bbH^2} = 0, \lim_{k\to \infty}\|S\Ato v_k\|_{ \bbH^2} = 0\tn{ and }   \|w_k\|_{\bbH^1} = 1.
  $$
 As $\Ae$ is skew-adjoint, we find that there exists a orthogonal projection $\bbP$ on $\bbH_1\times \bbH_2$ with $\ker(\bbP) = \ker(\Ae),$ $\Ima(\bbP) = \overline{\Ima(\Ae)}$, and $\Ae \bbP = \bbP\Ae$ on $\dom{\Ae}$. The sequence $((1-\bbP)w_k )_{k\in \N}\subset \ker(\Ae)$ then satisfies
  \begin{align*}
      \lim_{k\to \infty}\|(1-\bbP) w_k\|_{\bbH^1\times \bbH^2} &= \lim_{k\to \infty}\frac{1}{|\lambda|}\|(i\lambda-\Ae) (1-\bbP)w_k\|_{\bbH^1\times \bbH^2}\\&=\lim_{k\to \infty}\frac{1}{|\lambda|}\|(1-\bbP)(i\lambda-\Ae) w_k\|_{\bbH^1\times \bbH^2}\\&\leq\lim_{k\to \infty}\frac{1}{|\lambda|}\|(i\lambda-\Ae) w_k\|_{\bbH^1\times \bbH^2} = 0.
  \end{align*} 
  Hence, the sequence $(\bbP w_k )_{k\in \N}\subset \overline{\Ima(\Ae)}$ satisfies $\lim_{k\to \infty}\|\bbP w_k\|_{\bbH^1\times \bbH^2} =1$ and $$\lim_{k\to \infty}\|(i\lambda-\Ae) \bbP w_k\|_{\bbH^1\times \bbH^2} =\lim_{k\to \infty}\|\bbP(i\lambda-\Ae) w_k\|_{\bbH^1\times \bbH^2} = 0.$$
  This means that $\bbP w_k$ is an approximate eigenfunction of $$\Ae\restrict{\overline{\Ima(\Ae)}}:\dom{\Ae}\cap \overline{\Ima(\Ae)}\subset \overline{\Ima(\Ae)}\to \overline{\Ima(\Ae)}$$ to the approximate eigenvalue $i\lambda$.
We note that $(\bbP w_k)_{k\in \N}$ is bounded with respect to the graph norm of $\Ae$.

In the following, we use that $\Ae\restrict{\overline{\Ima(\Ae)}} $  has compact resolvent. This is valid, as is shown in Lemma~\ref{lem:Aextcompactresolvent}. Under this assumption, we find that there exists a subsequence $(\bbP w_{k_j})_{j\in \N}$ of $(\bbP w_{k})_{k\in \N}$ which converges to an element $u \in \overline{\Ima(\Ae)}$.
In particular, 
$$\Ae \bbP w_{k_j} = i\lambda-(i\lambda-\Ae)\bbP w_{k_j}$$ converges to $i\lambda u$ for $j\to \infty$, proving that $u$ is an eigenfunction of $\Ae$ with eigenvalue $i\lambda$ due to the closedness of $\Ae.$
Moreover, the rear component of 
$$\bbP w_k = \bbP \begin{pmatrix}
    v_k \\ S\Ato v_k
\end{pmatrix} $$ converges to $0$, since $S\Ato v_k$ converges to zero by assumption and as every component of $(1-\bbP)w_k$ converges to zero. Thus,  $u =(u_1,u_2,u_3,u_4)  \in \dom{\Ae}\cap \overline{\Ima(\Ae)}$ satisfies $u_4 = 0$.

By the construction of $\Ae$ from $D_D$, there hence exists an element $x = (x_1,x_2)\in \dom{D_D} = H^1_0(\Omega)\times H^\di(\Omega)$  
 with $$x_1\restrict{\Omega_w} = u_1,~x_1\restrict{\Omega_h} = u_3,~x_2\restrict{\Omega_w} = u_2,~x_2\restrict{\Omega_h} = 0$$ and $$D_Dx = \begin{pmatrix}
     \di x_2 \\\gr x_1
 \end{pmatrix} = i\lambda \begin{pmatrix}
     x_1\\x_2
 \end{pmatrix}.$$

We deduce $\Delta x_1 + \lambda^2 x_1 = 0$ and $\gr x_1 \restrict{\Omega_h} = 0$. The analyticity of eigenfunction of the Laplacian, see, e.g.,  Theorem 1.2 in Part 3 of \cite{friedmanpde}, then guarantees that $x_1$ is constant on the whole domain $\Omega$, which together with $\tr(x_1) = 0$ yields $x_1 = 0 = x_2$. However, this contradicts the fact that $u$ is non-zero and hence shows that our initial assumption $i\lambda \in \sigma(A_S)$ must be incorrect.
\end{proof}
\begin{lemma}\label{lem:Aextcompactresolvent}
  
As claimed in the previous proof, $$\Ae\restrict{\overline{\Ima(\Ae)}}:\dom{\Ae}\cap \overline{\Ima(\Ae)}\subset \overline{\Ima(\Ae)}\to \overline{\Ima(\Ae)}$$ has compact resolvent.
\end{lemma}
 \begin{proof}
By the construction of $\Ae$ from $D_D$, we have  
 $$\dom{\Ae} \coloneq \left\{\begin{pmatrix}
       x_1\restrict{\Omega_w}\\x_2\restrict{\Omega_w}\\x_1\restrict{\Omega_h}\\x_2\restrict{\Omega_h}
    \end{pmatrix}~\middle|~x = \begin{pmatrix}
        x_1\\x_2
    \end{pmatrix}\in \dom{D_D} = H^1_0(\Omega)\times H^\di(\Omega)\right\}$$ and
    \begin{align*}
        \Ae \begin{pmatrix}
       x_1\restrict{\Omega_w}\\x_2\restrict{\Omega_w}\\x_1\restrict{\Omega_h}\\x_2\restrict{\Omega_h}
    \end{pmatrix} = \begin{pmatrix}
       (D_Dx)_1\restrict{\Omega_w}\\(D_Dx)_2\restrict{\Omega_w}\\(D_Dx)_1\restrict{\Omega_h}\\(D_Dx)_2\restrict{\Omega_h}
    \end{pmatrix}.
    \end{align*}
    In particular, $(  x_1\restrict{\Omega_w},x_2\restrict{\Omega_w},x_1\restrict{\Omega_h},x_2\restrict{\Omega_h})$ is an element of $\overline{\Ima(\Ae)} = \ker(\Ae)^\perp$
    if and only if $x = (x_1,x_2)$ is an element of $$\overline{\Ima(D_D)} = \ker(D_D)^\perp = ({0}\times \{v\in H^\di(\Omega)~|~\di (v) = 0\})^\perp = L^2(\Omega)\times \gr H^1_0(\Omega).$$ 
    This uses the orthogonal Helmholtz decomposition 
    $$L^2(\Omega) = \gr H^1_0(\Omega) \oplus  \{v\in H^\di(\Omega)~|~\di (v) = 0\}.$$

    Now assume that  $(  x_1^k\restrict{\Omega_w},x_2^k\restrict{\Omega_w},x_1^k\restrict{\Omega_h},x_2^k\restrict{\Omega_h})_{k \in \N}$ is an arbitrary, bounded (with respect to the graph norm) sequence in $$\dom{\Ae\restrict{\overline{\Ima(\Ae)}}},$$ i.e., $ (x_1^k,x_2^k)_{k\in \N}$ is a sequence in 
    $$\dom{D_D}\cap\overline{\Ima(D_D)} = H^1_0(\Omega)\times (\gr H^1_0(\Omega)\cap H^\di(\Omega))$$ with 
    $$\sup_{k\in \N} \left\|x_1^k \right\|_{H^1(\Omega)}+\left\|x_2^k\right\|_{H^\di(\Omega)}<\infty.$$

To show that $\Ae\restrict{\overline{\Ima(\Ae)}}$ has compact resolvent, it thus suffices to find a subsequence of $(x_1^k,x_2^k)_{k\in \N}$ which converges in $L^2(\Omega)\times L^2(\Omega)^n$. As $H^1(\Omega)$ embeds compactly into $L^2(\Omega)$, it actually suffices to find a subsequence of $(x_2^k)_{k\in \N}$ which converges in $ L^2(\Omega)^n$. 

We note that since $(x_2^k)_{k\in \N}\subset H^\di(\Omega)\cap \gr H^1_0(\Omega)$, there exists a sequence of potentials $(p^k)_{k\in \N}\subset H^\Delta_0(\Omega)$ with $\gr p^k = x_2^k, k\in \N$. The Poincaré inequality yields that $(p^k)_{k\in \N}$ is even a bounded sequence in $H^\Delta_0(\Omega)$. 
As $H^1_0(\Omega)$ embeds compactly into $L^2(\Omega)$ by the Rellich-Kondrachov Theorem, we find a subsequence $(p^{k_j})_{j\in \N}$ which converges in $L^2(\Omega)$ and therefore
\begin{align*}
    \lim_{j\to \infty}\left\|x^{k_j}_2\right\|_{L^2(\Omega)}  &= \lim_{j\to \infty} \io  |\gr p^{k_j}|^2\dx x 
    =  \lim_{j\to \infty} \io  |\Delta p^{k_j}||p^{k_j}| \dx x
    \\&\leq  \lim_{j\to \infty} \left\|\Delta p^{k_j}\right\|_{L^2(\Omega)}\left\|p^{k_j}\right\|_{L^2(\Omega)} \leq  \lim_{j\to \infty} \left\|p^{k_j}\right\|_{H^\Delta(\Omega)}\left\|p^{k_j}\right\|_{L^2(\Omega)} = 0,
\end{align*}
which completes the proof.
 \end{proof}

 \begin{remark}\label{rem:Dhascptresolvent}
     The proof of the previous lemma also implies that 
     $D_D\restrict{L^2(\Omega)\times \gr H^1_0(\Omega)}$ has compact resolvent.
 \end{remark}

We are now in a position to present the proof of the first main result Theorem~\ref{thm:mildsolutions}. \\

\textit{Proof of Theorem~\ref{thm:mildsolutions}.}
The kernel of $A_S$ has already been determined in Proposition~\ref{prop:kernelandspectrum}. 
Additionally, the Helmholtz decomposition with mixed boundary conditions, see  Theorem 4.2 in \cite{hodgeonforms}, yields $$\{w\in H^\di(\Omega_w)~|~ \di (w) = 0, \trg_N(w)=0\}^\perp = \gr H^1_{\Gamma_w}(\Omega_w)$$ and thus \begin{align*}
    \ker(A_S)^\perp = L^2(\Omega_w)\times \gr H^1_{\Gamma_w}(\Omega_w) \times L^2(\Omega_h).
\end{align*}

We have also already verified that $A_S$ generates a contractive semigroup in Lemma~\ref{lem:contractive}, which we denote by $T_S(t), t\geq 0$ in the following. It remains to show that $T_S(t)$ is strongly convergent and that $\TS(t) \coloneq T_S(t)\restrict{\ker(A_S)^\perp}$ is well-defined and strongly stable. 

For this, we first have to verify that $\ker(A_S)^\perp$ is invariant under $T_S(t)$. The semigroup $T_S(t)$ is contractive and hence mean-ergodic, see, e.g.,  Corollary V.4.6 and Example V.4.7 in \cite{engel2000one}.
The contractivity also implies that the mean ergodic projection 
$$P:\bbH_1 \to \bbH_1, x\mapsto \lim_{t\to \infty } \frac{1}{t}\int_0^s T_S(t)x\dx s $$
 exists and has norm $\leq 1$, so $P$ is in fact an orthogonal projection.
Lemma~V.4.4 in \cite{engel2000one} then yields the equality $$\ker(A_S)^\perp = \overline{\Ima(P)}^\perp = \ker(P),$$ and $\ker(P)$ is an invariant subspace under the semigroup $T_S(t)$ since $P$ commutes with $T_S(t)$.
Thus, the restricted semigroup $\TS$ is well-defined, contractive and has generator $\AS \coloneq A_S\restrict{\ker(A_S)^\perp}$ on $\ker(A_S)^\perp$.

We find 
$\sigma(A_S^\perp)\cap i\R \subset \sigma(A_S)\cap i\R =\{0\},$ but $0$ is clearly not an eigenvalue of $A^\perp_S$. 
The ABLV-Theorem, see Theorem V.2.21 and Corollary V.2.22 in \cite{engel2000one}, therefore implies that $\TS(t)$ is strongly stable, which is what we wanted to show. 
Combining this with the fact that $T_S(t)\restrict{\ker(A_S)}$ is constant, we indeed find that $T_S(t)$ is strongly convergent. 
This completes the proof of Theorem~\ref{thm:mildsolutions}. \qed \\

\section{Conditions for Non-Uniform Decay Rates}
\label{sec:mainnonuniform}

Having established strong asymptotic stability of the contractive semigroup $\TS$ generated by $\AS$ in Section~\ref{sec:mainstability}, we now focus on finding specific rates for the convergence of $\|\TS(t)x\|_{\bbH_1}$ to zero. 
A uniform decay rate, i.e., a stability result of the form 
$$\lim_{t\to \infty}\|\TS(t)\|_{\bbH_1\to \bbH_1 } = 0,$$ would imply exponentially fast decay by the semigroup property. However, under some technical smoothness assumptions on $\Gamma_w,\gamma$, such a stability result on $\TS$ was shown to not hold on any Lipschitz interface triple $(\Omega,\Omega_w,\Omega_h)$ even in the simplest case $S\equiv 1$, see Chapter~6 in \cite{zuazualongtime}.

We thus restrict ourselves to investigating non-uniform decay rates.
As before, we denote by $\|\argument\|_A$ the graph norm of some operator $A$. 
A non-uniform decay rate on $\TS$ is an estimate of the form 
\begin{align*}
    \left\|\TS(t)x\right\|_{\bbH_1} \leq C_k d(t)^k \|x\|_{(\AS)^k} ~~ \forall t\geq0,  x\in \dom{(\AS)^k}, k \in \N,
\end{align*} where  $d:[0,\infty)\to[0,\infty)$ is some decreasing decay rate with $\lim_{t\to \infty} d(t) = 0 $ and $(C_k)_{k\in \N}$ are positive constants.

 Non-uniform decay rates are estimates on only classical solutions to the wave-heat system, as $\dom{\AS}$ is precisely the set of initial data $x$ for which $t\mapsto \TS(t)x$ is a classical solution.

\subsection{Resolvent Estimates} \label{sec:resolventestimates}

This subsection is concerned with finding equivalent conditions to the existence of non-uniform decay rates on $\TS.$

A general framework for finding (optimal) non-uniform decay rates by estimating the growth of the resolvent of the generator along the imaginary axis was already developed by Batty, Duyckaerts \cite{batty2008non} and Borichev, Tomilov \cite{borichev2010optimal}, building upon the work of Burq \cite{Burq} and Lebeau \cite{Lebeau}.
For convenience, we repeat their results here.
\begin{theorem}[Resolvent bounds and semigroup decay rates]
\label{thm:abstractresolventbounds}
Let $(T(t))_{t\geq 0}$ be a bounded semigroup with generator $A$ on a Hilbert space $H$  such that $\sigma(A)\cap i\R = \varnothing.$
\begin{enumerate}
    \item \tn{(Batty, Duyckaerts, Burq, Lebeau)}
    The semigroup $T$ has logarithmic decay with rate $1/d,d>0$, i.e., 
    \begin{align}\label{eq:logdecaydef}
        \forall k \in \N ~ \exists C_k>0: ~ \|T(t)x\|_{H} \leq \frac{C_k}{\log(t+2)^{\frac{k}{d}}} \|x\|_{A^k} ~~ \forall t\geq0,  x\in \dom{A^k},
    \end{align}
    if and only if the generator $A$ has exponential resolvent bounds with power $d>0$, i.e., there exists $C>0$ such that
    \begin{align}
        \label{eq:expresolventbound}
        \|(\lambda i +A)\inv \|_{H \to H} \leq C \exp{(C|\lambda|^d)} ~~\forall \lambda\in \R, |\lambda|\bb1.
    \end{align}
    \item   \tn{(Borichev, Tomilov)}
    The semigroup $T$ has polynomial decay with rate $1/d,$  $d>0$, i.e., 
    \begin{align}\label{eq:polynomdecaydef}
        \forall k \in \N ~ \exists C_k>0: ~ \|T(t)x\|_{H} \leq \frac{C_k}{(1+t)^{\frac{k}{d}}} \|x\|_{A^k} ~~ \forall t\geq0,  x\in \dom{A^k},
    \end{align}
    if and only if the generator $A$ has polynomial resolvent bounds with power $d>0$, i.e., there exists $C>0$ such that
    \begin{align}
        \label{eq:polynomresolventbound}
        \|(\lambda i +A)\inv \|_{H \to H} \leq C |\lambda|^d ~~\forall \lambda\in \R, |\lambda|\bb 1.
    \end{align}
\end{enumerate}
    
\end{theorem}
We note that if one aims to show the resolvent bounds~\eqref{eq:expresolventbound},~\eqref{eq:polynomresolventbound}, it suffices to show the respective semigroup decay rates~\eqref{eq:logdecaydef},~\eqref{eq:polynomdecaydef} only for $k=1$. 
This follows, e.g., from observation~(1.5) in \cite{batty2008non}.\\

\textbf{Goal:}    
In the following sections, we aim to find polynomial and logarithmic decay rates on $\TS$. As $\TS,\AS$ satisfy the assumptions of Theorem~\ref{thm:abstractresolventbounds}, see Lemma~\ref{lem:contractive} and Proposition~\ref{prop:kernelandspectrum}, we specifically aim to find lower bounds of the form 
\begin{align}
    \label{eq:resolventestimatepolynomial2}
    \left\|\left(\lambda i +\AS  \right)f\right\|_{\bbH_1} \gtrsim p(\lambda) \|f\|_{\bbH_1}, ~~ |\lambda|\gg1,\lambda\in \R, f\in \dom{\AS},
\end{align} 
where $p:\R \to [0,\infty) $ is some weight function that decays at most exponentially fast to zero for $|\lambda| \to \infty$.\\

Here and in the following, we use the notation $X\approx Y,~ X\lesssim Y,~ X\gtrsim Y$ for estimates to mean that there exists a
constant $C > 0$, independent of the variables in the estimate, such that $X/C \leq Y \leq CX,$ $X \leq CY, CX\geq Y$, respectively.

\subsection{Reduction to the wave domain}\label{sec:reductiontowavedomain}
On our way to finding resolvent bounds on $\AS$, we first introduce  the following three auxiliary results, the first of which follows the proof of Lemma 3.1 in \cite{gluck2024stability} and the closure relation framework.
\begin{lemma}
  For each strictly  accretive $S\in \domS$, the estimate $$\|\Ato f\|^2_{\bbH_2} \lesssim \|f\|_{\bbH_1}\|(i \lambda +A_S)f\|_{\bbH_1}$$ holds for all $f \in \dom{\AS} $, $\lambda \in \R.$\label{lem:A21h}
\end{lemma}
\begin{proof} 
    Define the vector 
    $$z = \begin{pmatrix}
        f\\ S\Ato f
    \end{pmatrix},$$
    which is an element of $\dom{A_1}$ and satisfies $A_S f = A_1 z$ by the construction \eqref{eq:defofAS} of $A_S$.
    
    Using the strict accretivity bound $\re (\overline{v}^\top S(x)v)\gtrsim |v|^2$ for all $v\in \C^n$ and almost all $x\in \Omega_h$, we continue with the estimate
    \begin{align*}
        \|\Ato f\|^2_{\bbH_2} &\lesssim \re\sca{S \Ato f}{ \Ato v}_{\bbH_2}\\ &= \re\sca{S \Ato f}{\Ato f}_{\bbH_2} + \re\sca{f}{-i \lambda f}_{\bbH_1} \\ &= \re\sca{S \Ato f}{\Ato f}_{\bbH_2} + \re\sca{f}{-i \lambda f}_{\bbH_1}+ \re\sca{z}{-\Ae z}_{\bbH_1\times \bbH_2}.
    \end{align*} The last equality uses the skew-adjointness of $\Ae$ demonstrated in Lemma~\ref{lem:contractive}. We also note that $S(x)$ is invertible almost everywhere by the assumed strict accretivity of $S.$
    We close the argument via the computation
    \begin{align*}
        &\hspace{3ex}\re\sca{S \Ato f}{S\inv S \Ato f}_{\bbH_2} + \re\sca{f}{-i \lambda f}_{\bbH_1}+ \re\sca{z}{-\Ae z}_{\bbH_1\times \bbH_2}
        \\&=\re\sca{z}{\left(\begin{pmatrix}
            -i\lambda & 0\\
            0 & S\inv
        \end{pmatrix}-\Ae\right)z}_{\bbH_1\times \bbH_2}
        \\&=\re\sca{z}{\begin{pmatrix}
            -i\lambda f\\
         \Ato f
        \end{pmatrix}-\begin{pmatrix}
            A_1 z\\\Ato f
        \end{pmatrix}}_{\bbH_1\times \bbH_2}
        \\&=\re\sca{f}{-i\lambda f - A_1 z}_{\bbH_1}
        \\&=\re\sca{f}{-(i\lambda  + A_S) f}_{\bbH_1} \leq \|f\|_{\bbH_1}\|(i \lambda +A_S)f\|_{\bbH_1}.\qedhere
    \end{align*}
    
\end{proof}

The second auxiliary result roughly corresponds to the existence of a bounded extension operator from $H^{\di S \gr}_{\Gamma_h}(\Omega_h)$ to $\dom{\AS}$.
\begin{lemma} \label{lem:smallextensionoperator}
Let  $S\in \dom{S}$ be strictly accretive.
     For each $$(w_1,w_2,h_1)  \in \dom{\AS}= \dom{A_S} \cap \ker(A_S)^\perp,$$ there exist two bounded extensions $$u_1 \in H^1_{\Gamma_w}(\Omega_w),~~u_2 \in \gr H^\Delta_{\Gamma_w}(\Omega_w),$$ independent of $(w_1,w_2)$, satisfying 
        \begin{enumerate}
            \item $(u_1,u_2,h_1)  \in \dom{\AS}$,
            \item $\|u_1\|_{H^1(\Omega_w)} \lesssim\|h_1\|_{H^1(\Omega_h)},$
            \item $\|u_2\|_{L^2(\Omega_w)}\lesssim\|\gr h_1\|_{L^2(\Omega_h)}$, 
            \item $\|u_2\|_{H^\di(\Omega_w)} \lesssim\|S\gr h_1\|_{H^\di(\Omega_h)}.$
        \end{enumerate}
    
\end{lemma}
\begin{proof}  We recall $
    \ker(A_S)^\perp = L^2(\Omega_w)\times \gr H^1_{\Gamma_w}(\Omega_w) \times L^2(\Omega_h)$ and
    \begin{align*} \dom{A_S} &= \left\{
            \begin{pmatrix}
                w_1\\w_2\\h_1
            \end{pmatrix}
            \in H^1_{\Gamma_w}(\Omega_w)\times H^\di(\Omega_w)\times H^{\di S \gr}_{\Gamma_h}(\Omega_h)~ \right|
            \nonumber\\&\hspace{6ex}  \trg(w_1)=\trg(h_1), \trg_N(w_2)=-\trg_N(S \gr h_1)\left. \rule{0cm}{0.75cm}\right\},
\end{align*}
which combines to 
\begin{align} \dom{\AS} &= \left\{
            \begin{pmatrix}
                w_1\\w_2\\h_1
            \end{pmatrix}
            \in H^1_{\Gamma_w}(\Omega_w)\times \gr H^\Delta_{\Gamma_w}(\Omega_w)\times H^{\di S \gr}_{\Gamma_h}(\Omega_h)~ \right|
            \nonumber\\&\hspace{6ex}  \trg(w_1)=\trg(h_1), \trg_N(w_2)=-\trg_N(S \gr h_1)\left. \rule{0cm}{0.75cm}\right\}. \label{eq:characterizationofASperp}
\end{align}

    First, Proposition~\ref{prop:patching}.(i) and the zero boundary trace condition $\tr^{\Gamma_h}(h_1)=0$ guarantee that the extension of $h_1$ by zero to $\R^n\backslash \Omega_w$, which we denote by $h_1^0$, is an element of $H^1(\R^n\backslash \Omega_w)$. We then choose $$u_1 = E_{\Omega_w} (\tr_{\R^n\backslash \Omega_w} (h_1^0)),$$ where the extension operator $E_{\Omega_w}$ is the right inverse of the surjective trace operator $$\tr:H^1(\Omega_w)\to H^{1/2}(\dO_w)  $$ and whose existence follows from the open mapping theorem. By construction, $u_1$ satisfies 
    \begin{align*}
         \|u_1\|_{H^1(\Omega_1)} &\lesssim \left\|\tr_{\R^n\backslash \Omega_w} (h_1^0)\right\| _{H^{1/2}(\dO_w)} \lesssim  \|h_1\|_{H^1(\Omega_h)}, \\\trg(u_1) &= \trg (h_1^0) = \trg (h_1), \tn{ and }\\
         \trGw(u_1) &= \trGw (h_1^0) = 0,
    \end{align*}
     i.e., $u_1 \in H^1_{\Gamma_w}(\Omega_w)$ indeed satisfies the properties we require.

    For the construction of $u_2$, let $E^\di_{\Omega_h}:H^\di(\Omega_h)\to H^\di(\R^n)$ denote the $L^2$-bounded extension operator from Proposition~\ref{prop:hdivextensionop} and let $\bbP_\gamma$ again denote the Helmholtz projection in $L^2(\Omega_w)$ with range $$\{v \in H^\di(\Omega_w)~|~ \di (v) = 0, \trg_N(v)=0\}$$ and kernel $\gr H^1_{\Gamma_w}(\Omega_w)$, which exists by Theorem 4.2 in \cite{hodgeonforms}.
    
    We then consider the choice $$u_2  \coloneq  (1-\bbP_\gamma)\left((E^\di_{\Omega_h} S \gr h_1) \restrict{\Omega_w}\right).$$
    As $\di \circ \bbP_\gamma = 0$, we find that the function $u_2$
    is an element of 
    $$\gr H^1_{\Gamma_w}(\Omega_w) \cap H^\di(\Omega_w) = \gr H^\Delta_{\Gamma_w}(\Omega_w)$$ with  
    $$\|u_2\|_{L^2(\Omega_w)}\leq \left\|E^\di_{\Omega_h} S \gr h_1\right\|_{L^2(\Omega_w)}\lesssim\| S\gr h_1\|_{L^2(\Omega_h)}\lesssim \|\gr h_1\|_{L^2(\Omega_h)}$$  and 
    $$\| u_2\|_{H^\di(\Omega_w)} 
    \leq \left\| E^\di_{\Omega_h} S \gr h_1\right\|_{H^\di(\Omega_w)} \lesssim \| S \gr h_1\|_{H^\di(\Omega_w)}.$$
 Since $\trg_N\circ \bbP_\gamma$ is also identically zero, we even have
    $$\trg_N(u_2) = \trg_N\left((E^\di_{\Omega_h} S \gr h_1)\restrict{\Omega_w}\right) = -\trg_N(S \gr h_1).$$

    In total, this confirms $(u_1,u_2,h_1) \in \dom{\AS}$ as well as the claimed estimates
    \begin{align*}
        \|u_1\|_{H^1(\Omega_w)} &\lesssim\|\gr h_1\|_{L^2(\Omega_h)}, ~~
        \|u_2\|_{L^2(\Omega_w)}\lesssim \|\gr h_1\|_{L^2(\Omega_h)},\\
        \|u_2\|_{H^\di(\Omega_w)} &\lesssim \|S\gr h_1\|_{H^\di(\Omega_h)}.\qedhere
    \end{align*}
\end{proof}

The third auxiliary result is a version of Poincaré's inequality, where one replaces the zero boundary trace condition with control by the second order term $(\lambda+ \di S \gr) v$.

\begin{lemma}[Poincarè-type inequality]
\label{lem:poincare}
    Let $\Tilde{S}:L^2(\Omega)\to L^2(\Omega)$ be a bounded linear operator. We then have the estimate $$\|v\|_{L^2(\Omega)} \lesssim \|\gr v\|_{L^2(\Omega)}+\frac{1}{|\lambda|}\|(\lambda+ \di \Tilde{S} \gr) v\|_{L^2(\Omega)}$$ for all $v \in H^{ \di \Tilde{S} \gr}(\Omega)$ and all $\lambda \in \C$ with $|\lambda|\geq 1$. The implicit constant in this estimate may depend on $\Tilde{S}$.
\end{lemma}
We note that $H^{ \di \Tilde{S} \gr}(\Omega)$ is defined analogously to $H^{ \di S \gr}(\Omega)$ in \eqref{eq:defofdivSgr}.
\begin{proof}[Proof of Lemma \ref{lem:poincare}]
    Assume otherwise, there then exists a sequence $(v_k,\lambda_k)_{k\in\N} \subset H^{\di \Tilde{S} \gr}(\Omega)\times \C$ with $\|v_k\|_{L^2(\Omega)} =1$, $|\lambda_k| \geq 1$, and 
    \begin{align}
        \|\gr v_k\|_{L^2(\Omega)}+\frac{1}{|\lambda_k|}\|(\lambda_k+\di \Tilde{S} \gr) v_k\|_{L^2(\Omega)}\leq 1/k.\label{eq:terminvolvingdivSgr}
    \end{align} By potentially replacing $v_k$ with $\exp(i\theta )v_k, \theta \in [0,2\pi)$, we can assume $v_k^0  \coloneq  \frac{1}{|\Omega|}\io v_k \dx x $ to be real and non-negative.
    The Poincaré-Wirtinger inequality now yields \begin{align*}
        \lim_{k\to \infty}\left|1-(v_k^0)|\Omega|^{1/2}\right|&= \lim_{k\to \infty}\left|\|v_k\|_{L^2(\Omega)}-\|v_k^0\|_{L^2(\Omega)} \right|
        \\&\leq \lim_{k\to \infty}\|v_k-v_k^0\|_{L^2(\Omega)}
        \\& \lesssim \lim_{k\to \infty}\|\gr v_k\|_{L^2(\Omega)} = 0,
    \end{align*}
    from which we infer $\lim_{k\to \infty} v_k^0 = {|\Omega|^{-1/2}}$ and even $\lim_{k\to \infty}\|v_k-{|\Omega|^{-1/2}}\|_{L^2(\Omega)}=0.$
    Therefore, the sequence $(v_k)_{k\in \N}$ converges in $L^2(\Omega)$ to the constant function $|\Omega|^{-1/2}.$

    Substituting this into assumption~\eqref{eq:terminvolvingdivSgr} leads to 
    \begin{align*}
        0 &= \lim_{k\to \infty} \frac{1}{|\lambda_k|}\|(\lambda_k+\di \Tilde{S} \gr) v_k\|_{L^2(\Omega)} 
        \\&=\lim_{k\to \infty} \left\|v_k+\di \left(\frac{\Tilde{S} \gr v_k}{\lambda_k}\right)\right\|_{L^2(\Omega)} 
        \\&=\lim_{k\to \infty} \left\||\Omega|^{-1/2}+\di \left(\frac{\Tilde{S} \gr v_k}{\lambda_k}\right)\right\|_{L^2(\Omega)}.
    \end{align*}
    The above shows that the divergence of the sequence $(\frac{\Tilde{S} \gr v_k}{\lambda_k})_{k\in \N}$ converges to a non-zero (constant) function. This, however, contradicts the closedness of the divergence operator on $L^2(\Omega)$ and the assumption \begin{align*}
    \lim_{k \to \infty } \left\|\frac{\Tilde{S} \gr v_k}{\lambda_k}\right\|_{L^2(\Omega)} \lesssim \lim_{k\to \infty} \left\| \frac{ \gr v_k}{\lambda_k}\right\|_{L^2(\Omega)} = 0. &\qedhere
    \end{align*}
    
\end{proof}

Having established these three auxiliary results, we are now in a position to demonstrate that resolvent bounds on $\AS $ can be reduced to the wave domain, i.e., to functions that vanish on $\Omega_h$, see also Remark~\ref{rem:reductiontoD}.

\begin{proposition}[Reduction to the case $h_1=0$]\label{prop:reductiontohzero}
    Let $(\Omega,\Omega_w,\Omega_h)$ be a Lipschitz interface triple, let  $S\in \dom{S}$ be strictly accretive, and let $p:\R \to [0,\infty)$ be a weight function satisfying $p(\lambda)\lesssim |\lambda|$ for all $|\lambda|\bb 1.$ 
    If the generator $\AS$ of the coupled wave-heat system~\eqref{eq:1} satisfies the resolvent estimate
    \begin{align*}
        \|(\lambda i +\AS  )f\|_{\bbH_1} \gtrsim p(\lambda) \|f\|_{\bbH_1}, ~~ |\lambda|\gg1,\lambda\in \R
    \end{align*}
    for every $f = (w_1,w_2,h_1)  \in \dom{\AS}$ with $h_1=0$, then $\AS$ even satisfies the estimate
    \begin{align*}
        \|(\lambda i +\AS  )f\|_{\bbH_1} \gtrsim \left(\frac{p(\lambda)}{|\lambda|}\right)^2 \|f\|_{\bbH_1},~~ |\lambda|\gg1,\lambda\in \R
    \end{align*}
    for all $f \in \dom{\AS}$.
\end{proposition}
\begin{proof}
    Let $\lambda\bb 1$ be fixed, let $f = (w_1,w_2,h_1)  \in \dom{\AS}$ be arbitrary and let $u_1 \in H^1_{\Gamma_w}(\Omega),~u_2 \in \gr H^\Delta_{\Gamma_w}(\Omega)$ be the extensions constructed in Lemma~\ref{lem:smallextensionoperator}, which satisfy \begin{enumerate}
            \item $(u_1,u_2,h_1)  \in \dom{\AS}$,
            \item $\|u_1\|_{H^1(\Omega_w)} \lesssim\|h_1\|_{H^1(\Omega_h)},$
            \item $\|u_2\|_{L^2(\Omega_w)}\lesssim\|\gr h_1\|_{L^2(\Omega_h)}$,
            \item $\|u_2\|_{H^\di(\Omega_w)} \lesssim\|S\gr h_1\|_{H^\di(\Omega_h)}.$
        \end{enumerate}
        We can then decompose $f$ into a part with $h_1=0$ and a part bounded by $h_1$, namely 
        $$
            f   = \begin{pmatrix}
                w_1-u_1\\w_2-u_2\\0
            \end{pmatrix} + \begin{pmatrix}
                u_1\\u_2\\h_1
            \end{pmatrix}.
        $$
        Using the assumption and Lemma~\ref{lem:smallextensionoperator} (ii), we estimate 
        \begin{align*}
            p(\lambda)\|f\|_{\bbH_1} &\leq p(\lambda)\left\|\begin{pmatrix}
                w_1-u_1\\w_2-u_2\\0
            \end{pmatrix}\right\|_{\bbH_1}+ p(\lambda)\left\|\begin{pmatrix}
                u_1\\u_2\\h_1
            \end{pmatrix}\right\|_{\bbH_1}\\
            &\lesssim  \left\|(\lambda i+\AS )\begin{pmatrix}
                w_1-u_1\\w_2-u_2\\0
            \end{pmatrix}\right\|_{\bbH_1}+p(\lambda) \|h_1\|_{H^1(\Omega_h)}\\
            &\leq  \left\|(\lambda i+\AS )f\right\|_{\bbH_1}+\left\|(\lambda i+\AS )\begin{pmatrix}
                u_1\\u_2\\h_1
            \end{pmatrix}\right\|_{\bbH_1}+ p(\lambda)\|h_1\|_{H^1(\Omega_h)}.
        \end{align*}
    By the definition of $\AS$ and the lower bound $|\lambda|\gg1$, the second term can be estimated by 
    \begin{align*}
        &\hspace{3ex}\left\|(\lambda i+\AS )\begin{pmatrix}
                u_1\\u_2\\h_1
            \end{pmatrix}\right\|_{\bbH_1} = \left\|\begin{pmatrix}
                \lambda i u_1 +\di u_2\\\lambda i u_2 +\gr u_1\\\lambda i h_1+\di S \gr h_1
            \end{pmatrix}\right\|_{\bbH_1} \\
        &\leq |\lambda|\Big(\|u_1\|_{H^1(\Omega_w)}+\|u_2\|_{L^2(\Omega_w)}+\|h_1\|_{L^2(\Omega_h)}\Big) + \|\di S \gr h_1\|_{L^2(\Omega_h)}+\|u_2\|_{H^\di(\Omega_w)}\\
        &\lesssim |\lambda| \|h_1\|_{H^1(\Omega_h)} + \|S\gr h_1\|_{H^\di(\Omega_h)} \\
        &\lesssim |\lambda| \|h_1\|_{H^1(\Omega_h)} + \|\di S\gr h_1\|_{L^2(\Omega_h)}
        \\ &\lesssim |\lambda| \|h_1\|_{H^1(\Omega_h)} + \|(\lambda i+\di S\gr) h_1\|_{L^2(\Omega_h)}
        \\ & \leq |\lambda| \|h_1\|_{H^1(\Omega_h)} + \left\|(\lambda i+\AS ) f\right\|_{\bbH_1}.
    \end{align*}
Combining this with the previous estimate, we find 
\begin{align*}
   p(\lambda)\|f\|_{\bbH_1}&\lesssim \left\|(\lambda i+\AS )f\right\|_{\bbH_1} + (p(\lambda)+|\lambda|)\|h_1\|_{H^1(\Omega_h)}
   \\&\lesssim \left\|(\lambda i+\AS )f\right\|_{\bbH_1} + |\lambda|\|h_1\|_{H^1(\Omega_h)}.
\end{align*}

We recall $\Ato f = \gr h_1$ from the definition of $\Ato$, so combining Lemma~\ref{lem:A21h} and Lemma~\ref{lem:poincare} leads to 
\begin{align*}
     \|h_1\|_{H^1(\Omega_h)} &\lesssim \|\gr h_1\|_{L^2(\Omega_h)} + |\lambda|^{-1} \|(\lambda i+\di S \gr ) h_1\|_{L^2(\Omega_h)} \\
     &= \|\Ato  f\|_{\bbH_2} + |\lambda|^{-1} \|(\lambda i+\di S \gr ) h_1\|_{L^2(\Omega_h)}
     \\
     &\lesssim \|f\|_{\bbH_1}^{1/2}\left\|(\lambda i+\AS )f\right\|_{\bbH_1}^{1/2} + |\lambda|^{-1} \left\|(\lambda i+\AS )f\right\|_{\bbH_1}
\end{align*} 
and hence
\begin{align*}
   p(\lambda)\|f\|_{\bbH_1}& \lesssim  \left\|(\lambda i+\AS )f\right\|_{\bbH_1} + |\lambda| \|h_1\|_{H^1(\Omega_h)} \\
   &\lesssim\left\|(\lambda i+\AS )f \right\|_{\bbH_1} + \Big(|\lambda|^{2}\left\|(\lambda i+\AS )f\right\|_{\bbH_1}\Big)^{1/2}\|f\|_{\bbH_1}^{1/2}
\end{align*}
We may assume $p(\lambda) >0$, otherwise the claimed estimate is trivial.
Dividing by $p(\lambda)$ leads to 
\begin{align*}
  \|f\|_{\bbH_1}& \lesssim \frac{1}{p(\lambda)}\left\|(\lambda i+\AS )f\right\|_{\bbH_1} + \left(\frac{|\lambda|^2}{p(\lambda)^2}\left\|(\lambda i+\AS )f\right\|_{\bbH_1}\right)^{1/2}\|f\|_{\bbH_1}^{1/2}
\end{align*}
An application of Young's inequality and the assumption $p(\lambda)\lesssim |\lambda|$ yields the claimed estimate
\begin{align*}
  \|f\|_{\bbH_1}& \lesssim\frac{1}{p(\lambda)}\left\|(\lambda i+\AS )f\right\|_{\bbH_1} + \frac{|\lambda|^2}{p(\lambda)^2}\left\|(\lambda i+\AS )f\right\|_{\bbH_1}
  \\ &\lesssim \frac{|\lambda|^2}{p(\lambda)^2}\left\|(\lambda i+\AS )f\right\|_{\bbH_1}. \qedhere
\end{align*}
\end{proof}
\begin{remark}\label{rem:afterhzeroreductiob}
We note that \begin{align*}
        \left\|(\lambda i+\AS )f\right\|_{\bbH_1} \gtrsim p(\lambda) \|f\|_{\bbH_1},~~ |\lambda|\gg1,\lambda\in \R
    \end{align*}
    for all $f \in \dom{\AS}$ clearly implies the same for all $f=(w_1,w_2,h_1)\in \dom{\AS}$ with $h_1=0.$
    \label{rem:reductiontoD}
    Recalling characterization~\eqref{eq:characterizationofASperp} of $\dom{\AS}$, we find that this is further equivalent to 
    \begin{align} \label{eq:polynomialboundforblock}
        \left\| \left( \lambda i+\begin{pmatrix}
            0 & \di \\ \gr & 0
        \end{pmatrix} \right) \begin{pmatrix}
            w_1\\w_2
        \end{pmatrix}\right\|_{2} \gtrsim p(\lambda) \left\| \begin{pmatrix}
            w_1\\w_2
        \end{pmatrix} \right\|_2, ~~|\lambda|\gg1,\lambda\in \R
    \end{align} 
    for all $(w_1,w_2)  \in H^1_0(\Omega_w)\times \{w\in \gr H^\Delta_{\Gamma_w}(\Omega_w) ~|~ \trg_N(\gr w) = 0\}$, where $f$ corresponds to $(w_1,w_2,0)$.

    This statement is especially interesting due to it not only being equivalent  (up to some loss in the parameter $p(\lambda)$) to resolvent bounds on $\AS$, as Proposition~\ref{prop:reductiontohzero} shows, but also due to it being independent of the heat domain $\Omega_h$, and hence independent of the heat coefficient $S$. We can thus infer that logarithmic and polynomial decay of the wave-heat system is, in some sense, independent of both the heat domain and heat coefficient. For a precise statement, see Theorem~\ref{thm:domdependence} and Theorem~\ref{thm:logdecay}.
\end{remark}

\subsection{Domain-Invariance}\label{sec:domaininvariance}
Troughout the rest of the article, we denote by 
$D_M$ the mixed boundary wave equation generator
\begin{align}
    \begin{pmatrix}
        0 & \di \\ \gr & 0
    \end{pmatrix} : \dom{D_M} \subset L^2(\Omega_w)\times \gr H^1_{\Gamma_w}(\Omega_w) \to L^2(\Omega_w)\times \gr H^1_{\Gamma_w}(\Omega_w) \label{eq:defofDM}
\end{align}
with domain
\begin{align*}
    \dom{D_M}&= H^1_{\Gamma_w}(\Omega_w) \times \{v\in \gr H^\Delta_{\Gamma_w}(\Omega_w) ~|~ \trg_N(v)=0\}.
\end{align*}
In this notation, the sought after resolvent-type   estimate~\eqref{eq:polynomialboundforblock} can be written as 
\begin{align}\label{eq:polynomialboundforD}
    \|(\lambda i+D_M)w\|_2 \gtrsim p(\lambda)\|w\|_2, ~~|\lambda|\gg1,\lambda\in \R,
\end{align} for all $w=(w_1,w_2)  \in \dom{D_M}$ with $\trg(w_1)= 0.$

The domain- and coefficient-dependence of polynomial decay rates on $\TS$ is the main focus of this subsection. 
We restrict ourselves to the dependence of polynomial, not logarithmic, decay rates, as $\TS$ has slow logarithmic decay rates on every Lipschitz interface triple, see Theorem~\ref{thm:logdecay}.
Nonetheless, the methods presented here apply analogously to the study of domain- and coefficient-dependence of logarithmic decay rates on $\TS$.

\begin{definition}\label{def:similar}
\begin{enumerate}
    \item We call $\Phi: \R^n \to \R^n$ a similarity transform with scaling factor $s\in \R\backslash\{0\}$ if there exists some orthogonal matrix $U \in \R^{n\times n}$ and some translation vector $\tau \in \R^n$ such that 
\begin{align*}
    \Phi(x) = s U x +\tau ~\forall x\in \R^n. 
\end{align*}

\item We call an admissible Lipschitz pair $(\Omega'_w, \gamma')$  $s$-similar to a  admissible Lipschitz pair $(\Omega_w, \gamma)$, $s\in \R\backslash\{0\}$, if there exists a similarity transform $\Phi$ with scaling factor $s$ such that $\Phi(\Omega_w) = \Omega_w'$ and $\Phi(\gamma) = \gamma'$. 

More generally, we call two admissible Lipschitz pairs similar if one of them is $s$-similar to the other for some $s\in \R\backslash\{0\}$.
It is easy to see that being similar is an equivalence relation.

\item 
Let $(\Omega_w,\gamma)$ be an admissible Lipschitz pair. We call $(\Omega'_w, \gamma')$ a \textit{Lipschitz interface extension} of $(\Omega_w,\gamma)$ if $\Omega'_w\backslash \overline{\Omega_w}$ is  a Lipschitz domain and $(\Omega'_w, \gamma')$ is itself an admissible Lipschitz pair satisfying $$\Omega_w\subset\Omega'_w ~\tn{ and }~\Gamma_w=  \dO_w\backslash \overline{\gamma} \subset \partial \Omega'_w \backslash \overline{ \gamma'} = \Gamma'_w.$$ 

\end{enumerate}
\end{definition}

Intuitively, a Lipschitz interface extension of $(\Omega_w,\gamma)$ can be created by extending $\Omega_w$ through $\gamma$, while keeping the rest of the original boundary $\Gamma_w $ intact. See also Figure~\ref{fig:interfacextension}.

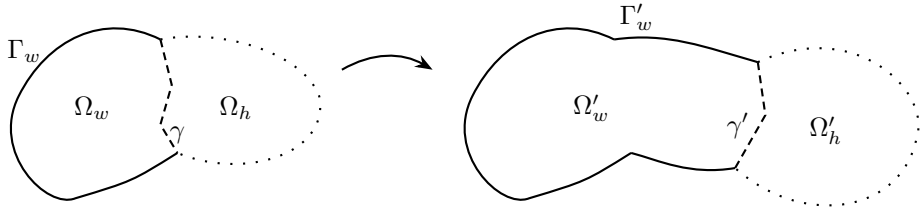
\begin{figure}[htbp]
    \centering
\begin{tikzpicture}[thick, scale=1]

    \begin{scope}[shift={(0,0)}]
        \coordinate (A) at (0, 1.5);
        \coordinate (B) at (0.15, 0.9);
        \coordinate (C) at (0, 0.375);
        \coordinate (D) at (0.225, 0);

        \draw (A) .. controls (-0.75, 1.875) and (-1.5, 1.5) .. (-1.875, 0.75) 
                 .. controls (-2.25, 0) and (-1.5, -0.75) .. (-1.125, -0.6)
                 .. controls (-0.375, -0.375) .. (D);
        
        \draw[densely dashed] (A) -- (B) -- (C) -- (D) node[above] {$\gamma$};

        \draw[loosely dotted] (A) .. controls (0.75, 1.65) and (1.875, 1.35) .. (2.1, 0.75)
                        .. controls (2.25, 0) and (1.125, -0.375) .. (D);

        \node at (-0.9, 0.6) {$\Omega_w$};
        \node at (1., 0.6) {$\Omega_h$};
        \node at (-1.8, 1.35) {$\Gamma_w$};
    \end{scope}

    \draw[->, >=Stealth] (2.4, 1.1) to [out=30, in=150] (3.6, 1.1);

    \begin{scope}[shift={(6,0)}] 
        \coordinate (A2) at (0, 1.5);
        \coordinate (D2) at (0.225, 0);
        
        \coordinate (E) at (1.9, 1.2); 
        \coordinate (F) at (1.6, -0.2);

        \draw (A2) .. controls (-0.75, 1.875) and (-1.5, 1.5) .. (-1.875, 0.75) 
                  .. controls (-2.25, 0) and (-1.5, -0.75) .. (-1.125, -0.6)
                  .. controls (-0.375, -0.375) .. (D2);
        \draw (A2) .. controls (0.7, 1.6) and (1.2, 1.4) .. (E);
        \draw (F) .. controls (1.0, -0.3) and (0.6, -0.1) .. (D2);
         
        \draw[densely dashed] (E) -- (2.0, 0.5) -- (F) node[left, pos=0.2] {\hspace{3mm}$\gamma'$};
        
        \draw[loosely dotted] (E) .. controls (3.2, 1.8) and (4.2, 0.8) .. (4.0, 0.2)
                        .. controls (3.8, -0.8) and (2.5, -1.0) .. (F);

        \node at (-0.3, 0.6) {$\Omega'_w$};
        \node at (2.8, 0.3) {$\Omega'_h$};
        \node at (0.3, 1.8) {$\Gamma'_w$};
    \end{scope}

\end{tikzpicture}
\caption{\textbf{Lipschitz interface extension.} The left panel shows the initial admissible Lipschitz pair $(\Omega_w,\gamma)$. The right panel illustrates the interface-extended, admissible Lipschitz pair $(\Omega_w',\gamma')$ after extending $\Omega_w$ through $\gamma$. In both cases, a potential Lipschitz domain $\Omega_h,\Omega'_h$ meeting $\Omega_w,\Omega'_w$ at the respective interface is indicated.}
    \label{fig:interfacextension}
\end{figure}

The following theorem presents the second main result of this paper.
\begin{theorem}[Domain- and Coefficient-Dependence]\label{thm:domdependence}
    Consider  two Lipschitz interface triples $(\Omega,\Omega_w,\Omega_h),(\Omega',\Omega'_w,\Omega'_h)$  with interfaces $\gamma,\gamma'$, where $(\Omega_w,\gamma)$ is similar to a Lipschitz interface extension of $(\Omega'_w,\gamma')$.
    
    If there exists some strictly  accretive heat coefficient $S \in \domS $ such that $\TS$ has polynomial decay  on $(\Omega,\Omega_w,\Omega_h)$ with rate $1/d, d>0$, 
  then $\TS$ has polynomial decay on $(\Omega',\Omega'_w,\Omega'_h)$ with rate $1/(2d+2)$ for all strictly  accretive $S \in \domS .$  
\end{theorem}

In particular, this theorem shows that polynomial decay rates on $\TS$ stay preserved, even when one changes the heat coefficients $S$ or increases the size of the interface $\gamma.$

On our way to proving Theorem~\ref{thm:domdependence}, we first demonstrate how similarity affects the sought after estimate~\eqref{eq:polynomialboundforD}. 

\begin{lemma} \label{lem:similar}
    Consider two $s$-similar admissible Lipschitz pairs $(\Omega_w, \gamma),(\Omega'_w, \gamma')$, $s\neq 0$, and some weight function $p:\R \to [0,\infty)$.
Assume that the mixed boundary wave equation generator $D_M$ has resolvent-type estimates on $\Omega'_w$, i.e., \begin{align*}
    \|(\lambda i+D_M)w\|_{L^2(\Omega'_w)} \geq p(\lambda)\|w\|_{L^2(\Omega'_w)}, ~~|\lambda|\gg1,\lambda\in \R
\end{align*} for all $w=(w_1,w_2)  \in \dom{D_M(\Omega'_w)}$ with $\tr^{\gamma'}(w_1)= 0,$ 
then
$D_M$ has resolvent-type estimates \begin{align*}
    \|(\lambda i+D_M)w\|_{L^2(\Omega_w)}\geq |s|p\left(\frac{\lambda}{s}\right)\|w\|_{L^2(\Omega_w)}, ~~|\lambda|\gg1,\lambda\in \R
\end{align*}  for all $w=(w_1,w_2)  \in \dom{D_M(\Omega_w)}$ with $\trg(w_1)= 0.$
\end{lemma}
\begin{proof}
    We first collect properties of how similarity transforms interact with $\di,\gr$, trace operators, and $L^2$-norms. Let $\Phi:\R^n \to \R^n, x \mapsto  sUx +\tau$ be a similarity transform with $\Phi(\Omega_w) = \Omega_w', \Phi(\gamma) = \gamma'$. We observe, for $f\in H^1(\Omega'_w), F\in H^\di(\Omega'_w), x \in \Omega_w$, the identities
    \begin{enumerate}
       
        \item $\di(U^\top F \circ \Phi)(x) = s(\di F)(\Phi(x))$,
        \item $\gr (f\circ \Phi)(x) = s U^\top (\gr f)(\Phi (x))$,
         \item $\|f\circ \Phi\|_{L^2(\Omega_w)} = |s|^{-n/2}\|f\|_{L^2(\Omega'_w)}  , $ 
        \item $\trg(f \circ \Phi) = \tr^{\gamma'}(f)\circ \Phi,$
        \item $\trg_N(U^\top F \circ \Phi) = \tr^{\gamma'}_N(F)\circ \Phi.$
    \end{enumerate}
    We note that the last identity is only well-defined if $\tr_N^{\gamma'}(F)$ is an $L^2(\gamma')$-function, not just an $\hnehz(\gamma')$-distribution. However, this is no obstruction, as every element $w\in \dom{D_M(\Omega_w)}$ satisfies $\trg_N(w_2)=0.$
    
    We observe that if $w = (w_1,w_2) $ is an element of $\dom{D_M(\Omega'_w)}$ with $\tr^{\gamma'}(w_1) = 0$, then
    $$\Psi(w)  \coloneq  \begin{pmatrix}
        \Psi_1(w)\\\Psi_2(w)
    \end{pmatrix}  \coloneq \begin{pmatrix}
        w_1 \circ \Phi\\ U^\top w_2\circ \Phi
    \end{pmatrix} $$ is an element of $\dom{D_M(\Omega_w)}$ with $\trg(\Psi_1(w)) = 0.$ As $U$ is invertible, we find that $\Psi$ is a linear bijection between 
    $$ \Xi  \coloneq  \{w \in \dom{D_M(\Omega'_w)}~|~ \tr^{\gamma'}(w_1) = 0\}$$ and $$ \{w \in \dom{D_M(\Omega_w)}~|~ \tr^{\gamma}(w_1) = 0\}.$$

    Now assume that we have the resolvent-type bounds 
\begin{align*}
    \|(\lambda i+D_M)w\|_{L^2(\Omega'_w)}\geq p(\lambda)\|w\|_{L^2(\Omega'_w)}, ~~|\lambda|\gg1,\lambda\in \R,w\in \Xi,
\end{align*}
it suffices to show \begin{align*}
    \|(\lambda i+D_M)\Psi (w)\|_{L^2(\Omega_w)}\geq p(\lambda)\|\Psi (w)\|_{L^2(\Omega_w)}, ~~|\lambda|\gg1,\lambda\in \R,w\in \Xi.
\end{align*} The reverse implication then follows as similarity is an equivalence relation.

Using the identities~(i)--(v) from the beginning of this proof, we directly compute
\begin{align*}
     \|(\lambda i+D_M)\Psi (w)\|_{L^2(\Omega_w)} & = \left\|\begin{pmatrix}
         \lambda i w_1\circ \Phi +  \di(U^\top w_2 \circ \Phi)
         \\  \lambda i U^\top  w_2\circ \Phi +  \gr( w_1 \circ \Phi)
     \end{pmatrix}\right\|_{L^2(\Omega_w)}
     \\&=  |s|\left\|\begin{pmatrix}
         \frac{\lambda i}{s} w_1\circ \Phi +  \di(w_2) \circ \Phi
         \\ U^\top \left(\frac{\lambda i}{s}  w_2\circ \Phi +  \gr( w_1) \circ \Phi \right)
     \end{pmatrix}\right\|_{L^2(\Omega_w)}
     \\&=  |s|^{-\frac{n}{2}+1}\left\|\begin{pmatrix}
         \frac{\lambda i}{s} w_1+  \di(w_2) 
         \\ \frac{\lambda i}{s}  w_2+  \gr w_1
     \end{pmatrix}\right\|_{L^2(\Omega'_w)}
     \\&=  |s|^{-\frac{n}{2}+1}\left\| \left(\frac{\lambda i}{s} -D_M\right)w \right\|_{L^2(\Omega'_w)}
     \\& \geq p\left(\frac{\lambda}{s}\right)|s|^{-\frac{n}{2}+1}\left\|w \right\|_{L^2(\Omega'_w)}
    \\ &=  p\left(\frac{\lambda}{s}\right)|s|\left\|\Psi(w) \right\|_{L^2(\Omega_w)}\qedhere.
\end{align*}
\end{proof}

\begin{remark}\label{rem:sisconsideredconstant}
    The scaling factor $s\neq 0$ depends only on the domains $\Omega_w,\Omega'_w$ and can hence be considered constant. In particular, for $p\in \R$ and under the assumptions of Lemma~\ref{lem:similar}, we find that the resolvent-type estimate \begin{align*}
    \|(\lambda i+D_M)w\|_{L^2(\Omega'_w)} \gtrsim |\lambda|^{p}\|w\|_{L^2(\Omega'_w)}, ~~|\lambda|\gg1,\lambda\in \R
\end{align*} for all $w=(w_1,w_2)  \in \dom{D_M(\Omega'_w)}$ with $\tr^{\gamma'}(w_1)= 0$ is equivalent to the resolvent-type estimate \begin{align*}
    \|(\lambda i+D_M)w\|_{L^2(\Omega_w)}\gtrsim |\lambda|^{p}\|w\|_{L^2(\Omega_w)}, ~~|\lambda|\gg1,\lambda\in \R
\end{align*}  for all $w=(w_1,w_2)  \in \dom{D_M(\Omega_w)}$ with $\trg(w_1)= 0.$
\end{remark}

Secondly, we determine how Lipschitz interface extensions influence the sought after resolvent-type estimate~\eqref{eq:polynomialboundforD}. 

\begin{proposition}\label{prop:interfacextension}
Let $p:\R \to [0,1]$ be a weight function, 
        let $(\Omega_w,\gamma)$ be an admissible Lipschitz pair and let $(\Omega'_w, \gamma')$ be a Lipschitz interface extension of $(\Omega_w,\gamma)$. If the mixed boundary wave equation generator $D_M$ has resolvent-type estimates  \begin{align*}
    \|(\lambda i+D_M)w\|_{L^2(\Omega'_w)} \geq p(\lambda)\|w\|_{L^2(\Omega'_w)}, ~~|\lambda|\gg1,\lambda\in \R,
\end{align*} for all $w=(w_1,w_2)  \in \dom{D_M(\Omega'_w)}$ with $\tr^{\gamma'}(w_1)= 0,$ 
then
$D_M$ has the same resolvent-type estimates on $\Omega_w$, i.e., \begin{align*}
    \|(\lambda i+D_M)w\|_{L^2(\Omega_w)}\geq p(\lambda)\|w\|_{L^2(\Omega_w)}, ~~|\lambda|\gg1,\lambda\in \R,
\end{align*}  for all $w=(w_1,w_2)  \in \dom{D_M(\Omega_w)}$ with $\trg(w_1)= 0.$

\end{proposition}
\begin{proof}
    Let $w=(w_1,w_2) \in \dom{D_M(\Omega_w)}$ be arbitrary with $\trg(w_1) = 0$. We  investigate whether the extension of $w$ by zero to $\Omega'_w \supset \Omega_w$, which we denote by $w'$, is again an element of $\dom{D_M(\Omega'_w)}$  with $\tr^{\gamma'}(w'_1) = 0$.
    
    We first recall 
    $$\dom{D_M(\Omega_w)} =H^1_{\Gamma_w}(\Omega_w) \times \left\{v\in \gr H^1_{\Gamma_w}(\Omega_w) \cap H^\di(\Omega_w) ~|~ \trg_N(v)=0\right\}.$$
    Since $\trg(w_1) = 0 = \trg_N(w_2)$, Proposition~\ref{prop:patching} guarantees that $w'$ is again an element of $H^1(\Omega'_w)\times H^\di(\Omega'_w)$. As $w'$ is identically zero on $\Omega'_w\backslash \overline{\Omega_w}$ and hence has zero trace on $$\partial \left(\Omega'_w \backslash \Omega_w \right)  \supset \left(\Gamma'_w \backslash \Gamma_w \right) \cup \gamma',$$  we additionally find $\tr^{\Gamma'_w}(w'_1) = 0$ and $\tr^{\gamma'}_N(w'_2)=0.$
    In total, the only property missing to show 
     $$w'\in \dom{D_M(\Omega'_w)} =H^1_{\Gamma'_w}(\Omega'_w) \times \left\{v\in \gr H^1_{\Gamma'_w}(\Omega'_w) \cap H^\di(\Omega'_w) ~|~ \tr^{\gamma'}_N(v)=0\right\}$$ is the inclusion $w'_2 \in \gr H^1_{\Gamma'_w}(\Omega'_w)$.
    Unfortunately, $w'_2$ is generally not even an element of $\gr H^1(\Omega'_w)$.

    Instead, we project the second component $w'_2$ onto this space. For this, let $\bbP_{\gamma'}$ denote the Helmholtz projection in $L^2(\Omega'_w)$ with range $$\left\{v \in H^\di(\Omega'_w)~|~ \di (v) = 0, \tr^{\gamma'}_N(v)=0\right\}$$ and kernel $\gr H^1_{\Gamma'_w}(\Omega'_w)$,   see  Theorem 4.2 in \cite{hodgeonforms}.
    We then define 
    $$u  \coloneq  \begin{pmatrix} w'_1 \\ (1-\bbP_{\gamma'})w'_2\end{pmatrix},$$ 
    which is indeed an element of $\dom{D_M(\Omega'_w)}$, as $\bbP_{\gamma'}w'_2$ has zero divergence and normal trace on $\gamma'$. 
    
    We can now apply the assumption to $u$, which yields
     \begin{align*}
        \|(\lambda i+D_M)u\|_{L^2(\Omega'_w)} \geq p(\lambda)\|u\|_{L^2(\Omega'_w)}, ~~|\lambda|\gg1,\lambda\in \R.
    \end{align*}
    On the other hand, we find
    \begin{align*}
        \|(\lambda i+D_M)w\|_{L^2(\Omega_w)}^2&= \left\|(\lambda i+D_M)w'\right\|_{L^2(\Omega'_w)}^2\\
        & = \left\|\lambda i w'_1 +\di (w'_2)\right\|_{L^2(\Omega'_w)}^2+\left\|\lambda i w'_2 + \gr w'_1\right\|_{L^2(\Omega'_w)}^2
        \\&= \|\lambda i u_1 +\di u_2\|_{L^2(\Omega'_w)}^2+\left\|\lambda i (u_2 +\bbP_{\gamma'} w'_2) + \gr u_1\right\|_{L^2(\Omega'_w)}^2
    \end{align*}
    for $|\lambda|\bb 1,\lambda \in \R.$
    Keeping in mind the orthogonality of the kernel and range of $\bbP_{\gamma'}$, as well as the fact that $\gr u_1 $ is an element of $\gr H^1_{\gamma'_w}(\Omega'_w)  = \ker(\bbP_{\gamma'})$, we find
    \begin{align*}
        &\hspace{3ex}\|\lambda i u_1 +\di u_2\|_{L^2(\Omega'_w)}^2+\left\|\lambda i (u_2 +\bbP_{\gamma'} w'_2) + \gr u_1\right\|_{L^2(\Omega'_w)}^2
        \\& = \|\lambda i u_1 +\di u_2\|_{L^2(\Omega'_w)}^2+\|\lambda i u_2 + \gr u_1\|_{L^2(\Omega'_w)}^2+\left\|\bbP_{\gamma'} w'_2\right\|_{L^2(\Omega'_w)}^2
        \\&=  \|(\lambda i+D_M)u\|_{L^2(\Omega'_w)}^2 +\left\|\bbP_{\gamma'} w'_2\right\|_{L^2(\Omega'_w)}^2
        \\& \geq p(\lambda)\|u\|_{L^2(\Omega'_w)}^2 +\left\|\bbP_{\gamma'} w'_2\right\|_{L^2(\Omega'_w)}^2
        \\&\geq p(\lambda) \|w'\|_{L^2(\Omega'_w)} ^2
        = p(\lambda) \|w\|_{L^2(\Omega_w)}^2.
    \end{align*} The penultimate step uses the assumption $p\leq 1$.  We conclude 
     \begin{align*}
    \|(\lambda i+D_M)w\|_{L^2(\Omega'_w)} &\geq p(\lambda)\|w\|_{L^2(\Omega'_w)}, ~~|\lambda|\gg1,\lambda\in \R.   \qedhere
\end{align*}
\end{proof}
With this, we are now able to prove that polynomial decay rates of the wave-heat 
semigroup $\TS$ behave monotonically with respect to $(\Omega_w,\gamma)$, and are independent of $S$ and $\Omega_h.$

\begin{proof}[Proof of Theorem~\ref{thm:domdependence}]
    Let $(\Omega,\Omega_w,\Omega_h),(\Omega',\Omega'_w,\Omega'_h)$ be as in the statement of the Theorem.
    Additionally, let $\TS,\AS$ denote the wave-heat semigroup and its generator on $(\Omega,\Omega_w,\Omega_h)$, while $(\TS)',(\AS)'$ denote the wave-heat semigroup and its generator on $(\Omega',\Omega'_w,\Omega'_h)$. Similarly, we denote by $D_M,D'_M$ the operator $D_M$ on  $(\Omega,\Omega_w,\Omega_h),(\Omega',\Omega'_w,\Omega'_h)$, respectively.
    
    Assume $\TS$ possess polynomial decay with rate $1/d, d>0$ for some strictly  accretive $S \in \domS $. 
    By Theorem~\ref{thm:abstractresolventbounds}.(ii), this implies the resolvent bound 
    $$\|(\lambda i +\AS ) f\|_{\bbH_1} \gtrsim |\lambda|^{-d}\|f\|_{\bbH_1},~~|\lambda|\gg1,\lambda\in \R, f\in \dom{\AS}.$$
    As in Remark~\ref{rem:afterhzeroreductiob}, we restrict to those $f=(w_1,w_2,h_1)\in \dom{\AS}$ with $h_1=0$, which yields
    \begin{align*}
    \|(\lambda i+D_M)w\|_2 \gtrsim |\lambda|^{-d}\|w\|_2, ~~|\lambda|\gg1,\lambda\in \R
\end{align*} for all $w=(w_1,w_2)  \in \dom{D_M}$ with $\trg(w_1)= 0,$ see also the beginning of Subsection~\ref{sec:domaininvariance}.
Since $(\Omega_w,\gamma)$ is similar to a Lipschitz interface extension of $(\Omega'_w,\gamma')$, Remark~\ref{rem:sisconsideredconstant} and Proposition~\ref{prop:interfacextension} imply \begin{align*}
    \|(\lambda i+D'_M)w'\|_2 \gtrsim |\lambda|^{-d}\|w'\|_2, ~~|\lambda|\gg1\lambda\in \R,
\end{align*} for all $w'=(w'_1,w'_2)  \in \dom{D'_M}$ with $\tr^{\gamma'}(w_1)= 0.$
After rewriting this estimate as a resolvent-type estimate on $(\AS)'$,
namely  $$\|(\lambda i +(\AS)') f'\|_{\bbH_1} \gtrsim |\lambda|^{-d}\|f'\|_{\bbH_1},~~|\lambda|\gg1,\lambda\in \R$$ for all $f' =(w'_1,w'_2,h'_1)\in \dom{(\AS)'}$ with $h'_1=0,$ we can apply Proposition~\ref{prop:reductiontohzero} to infer 
$$\|(\lambda i +(\AS)') f'\|_{\bbH_1} \gtrsim |\lambda|^{-(2d+2)}\|f'\|_{\bbH_1},~~|\lambda|\gg1,\lambda\in \R, f'\in \dom{(\AS)'}.$$
Lastly, Theorem~\ref{thm:abstractresolventbounds}.(ii) yields the claimed 
 polynomial decay of  $(\TS)'$ with rate $1/(2d+2)$  on $(\Omega',\Omega'_w,\Omega'_h)$ for each strictly  accretive $S\in \domS.$ 
\end{proof}

\section{Specific decay rates}
\label{sec:mainspecificrates}

\subsection{Logarithmic decay}
The logarithmic decay in the sense of \eqref{eq:logdecaydef} has already been established in the case $S\equiv 1$ by Fathallah \cite{ines}, 
using the implication $$\eqref{eq:expresolventbound}\implies \eqref{eq:logdecaydef}$$ as well as Carleman estimates near the interface $\gamma$, following the work of Bellassoued \cite{Bellassoued}. 

We extend the results of Fathallah \cite{ines} to non-trivial heat coefficients $S$ using the reduction to the wave domain established in Proposition~\ref{prop:reductiontohzero}.

\begin{theorem}[Logarithmic Decay]\label{thm:logdecay}
    Let $(\Omega,\Omega_w,\Omega_h)$ be a Lipschitz interface triple in $\R^n$ and let $S\in \domS $ be a strictly  accretive heat coefficient, i.e., 
    $$\re(\overline{v}^\top S(x) v) \gtrsim  |v|^2 $$ for all $x\in \C^n$ and almost all $x \in \Omega_h$. Then the semigroup $\TS$ generated by the generator $\AS$ of the coupled wave-heat system~\eqref{eq:1} has logarithmic decay with rate $1/d=1$, i.e.,  \begin{align*}
        \forall k \in \N ~ \exists C_k>0:  \|\TS(t)x\|_{\bbH_1} \leq \frac{C_k}{\log(t+2)^k} \|x\|_{(\AS)^k} ~ \forall t\geq0,  x\in \dom{(\AS)^k}.
    \end{align*}
   
\end{theorem}
\begin{proof}

   We begin by noting that in Theorem 1.1 of \cite{ines}, Fathallah has established the resolvent bound
   $$\|(i\lambda +\mathcal{A})\|_{H \to H} \leq C \exp(C |\lambda|)$$ for some $C>0$ and all $\lambda\in \R,|\lambda|\bb 1$, where
    $H$ is the Hilbert space
    $H= L^2(\Omega_w)\times H^1_{\Gamma_w}(\Omega_w) \times L^2(\Omega_h)$ and $\mathcal{A}$ is a modified version of the wave equation generator $\AS$, namely
    \begin{align*}
        \mathcal{A}\begin{pmatrix}
            w_1\\w_2\\h_1
        \end{pmatrix}  \coloneq  \begin{pmatrix}
            0&\Delta& 0\\
            1 & 0 & 0 \\
            0&0 & \Delta
        \end{pmatrix}\begin{pmatrix}
            w_1\\w_2\\h_1
        \end{pmatrix}= \begin{pmatrix}
            \Delta w_2 \\ w_1 \\ \Delta h_1
        \end{pmatrix} 
    \end{align*}with
    \begin{align*}
         \dom{\mathcal{A}} \coloneq  &\left\{ \begin{pmatrix}
            w_1\\w_2\\h_1
        \end{pmatrix} \in   H^1_{\Gamma_w}(\Omega_w)\times H^\Delta_{\Gamma_w}(\Omega_h)\times H^\Delta_{\Gamma_h}(\Omega_h)~\right|
        \\&\left.\phantom{\begin{pmatrix}
            .\\.\\.
        \end{pmatrix} }\trg(w_1) = \trg(h_1), \trg_N(\gr w_2) = -\trg_N(\gr h_1)  \right\}.
    \end{align*}
    Our main task will be transforming this resolvent bound on $\mathcal{A}$ into the resolvent bound~\eqref{eq:resolventestimatepolynomial2} on $\AS$.
    
    For this, let us first assume $\Gamma_w \neq \varnothing$ and let 
    \begin{align*}
     \begin{pmatrix}
        w_1 \\ w_2 \\0
    \end{pmatrix} \in \dom{\AS}&=\left\{
            \begin{pmatrix}
                w_1\\w_2\\h_1
            \end{pmatrix}
            \in H^1_{\Gamma_w}(\Omega_w)\times \gr H^\Delta_{\Gamma_w}(\Omega_w)\times H^{\di S \gr}_{\Gamma_h}(\Omega_h)~ \right| \nonumber\\&\hspace{6ex}  \trg(w_1)=\trg(h_1), \trg_N(w_2)=-\trg_N(S \gr h_1)\left. \rule{0cm}{0.75cm}\right\}
    \end{align*}
  be arbitrary. In particular, there exists a potential $p\in H^\Delta_{\Gamma_w}(\Omega_w)$ with $\gr p = w_2.$
  Additionally, we find $$\trg_N(\gr  p) = \trg_N(w_2) = -\trg_N(S \gr h_1) = 0 = -\trg_N(\gr h_1)$$ and hence that $(w_1,p,0) $ is an element of $\dom{\mathcal{A}},$
   allowing us to infer 
   \begin{align*}
       &\hspace{2ex}\left\|\lambda i w_1 + \Delta p\right\|_{L^2(\Omega_w)}+\left\|\lambda i p + w_1\right\|_{H^1(\Omega_w)} \approx \left\|\left(i \lambda+\mathcal{A}\right)\begin{pmatrix}
           w_1 \\ p \\0
       \end{pmatrix}\right\|_{H}\\& \gtrsim \exp(-C|\lambda|) \left\|\begin{pmatrix}
           w_1 \\ p \\0
       \end{pmatrix}\right\|_{H} 
       = \exp(-C|\lambda|)\left(\left\|w_1\right\|_{L^2(\Omega_w)}+\left\| p \right\|_{H^1(\Omega_w)}\right)
   \end{align*} for $\lambda\in\R,|\lambda|\bb 1.$
   
    Recalling $\tr^{\Gamma_w}(w_1) = \tr^{\Gamma_w}(p) =0$ and 
   \begin{align*}
       \left\|\left(i \lambda+\AS \right)\begin{pmatrix}
           w_1 \\ w_2 \\0
       \end{pmatrix}\right\|_{\bbH_1}&=\left\|\left(i \lambda+\begin{pmatrix}
           0 & \di \\ \gr &0
       \end{pmatrix}\right)\begin{pmatrix}
           w_1 \\ w_2
       \end{pmatrix}\right\|_{L^2(\Omega_w)},
   \end{align*}
    the Poincaré-inequality, which uses $\Gamma_w\neq \varnothing$, allows us to compute
    \begin{align*}
         \left\|\left(i \lambda+\AS \right)\begin{pmatrix}
           w_1 \\ w_2 \\0
       \end{pmatrix}\right\|_{\bbH_1} & \approx \|(i\lambda w_1 + \Delta p)\|_{L^2(\Omega_w)}+\|\gr (i\lambda p + w_1)\|_{L^2(\Omega_w)}
       \\& \approx \|(i\lambda w_1 + \Delta p)\|_{L^2(\Omega_w)}+\|i\lambda p + w_1\|_{H^1(\Omega_w)}
       \\& \gtrsim \exp(-C|\lambda|)\left(\left\|w_1\right\|_{L^2(\Omega_w)}+\left\| p \right\|_{H^1(\Omega_w)}\right)
       \\& \approx \exp(-C|\lambda|)\left(\left\|w_1\right\|_{L^2(\Omega_w)}+\left\| w_2 \right\|_{L^2(\Omega_w)}\right)
       \\& \approx \exp(-C|\lambda|)\left\|\begin{pmatrix}
           w_1\\w_2\\0
       \end{pmatrix}\right\|_{\bbH_1}.
    \end{align*}
    Applying Proposition~\ref{prop:reductiontohzero} to this, we deduce
    \begin{align*}
        \|(\lambda i +\AS  )f\|_{\bbH_1} \gtrsim \frac{\exp(-2C|\lambda|)}{|\lambda|^2} \|f\|_{\bbH_1},~~ |\lambda|\gg1,\lambda\in \R
    \end{align*}
    for all $f \in \dom{\AS}$. 
    For large $|\lambda|\bb 1$, one can find a (large) constant $C'>0$ such that $\exp(-2C|\lambda|)|\lambda|^{-2} \gtrsim \exp(-C'|\lambda|)$. 
    Theorem~\ref{thm:abstractresolventbounds}.(i) then yields the claimed logarithmic decay of $\TS$.

    In the case $\Gamma_w = \varnothing$, the domain $\Omega_h$ controls $\Omega_w$ in time (see Remark~\ref{rem:lions}), and we will show an even stronger, polynomial decay rate of $\TS$. This is done in Subsection~\ref{sec:GCC}, specifically Corollary~\ref{cor:GGCpolynomrate}.
\end{proof}

Having established a weak, but domain-independent logarithmic decay on the semigroup $\TS$, the remaining part of this paper will be about demonstrating stronger decay rates for $\TS$, i.e., demonstrating stronger than exponential resolvent-type estimates of the form \eqref{eq:polynomialboundforblock} for specific classes of domains.

\subsection{One-dimensional domains}
In \cite{zuazua1d,1dimoptimalrate}, decay rates for the semigroup $\TS$ on bounded $1$-dimensional Lipschitz interface triples were established in the case $S\equiv 1$, i.e., under the assumption that the heat coefficient $S$ is space independent. 
It was shown that $T^\perp_1$ has polynomial decay with rate $1/d=2$, and that this rate is optimal. 
Additionally, the same holds if the Dirichlet boundary condition on $\Gamma_w$ is replaced by a Neumann boundary condition. 

Similarly, the case of an unbounded heat domain $\Omega_w$ was considered in \cite{infiniteheatpart}, yielding a slightly slower, optimal, polynomial decay with rate $1/d = 1.$

Instead of applying Theorem~\ref{thm:domdependence} to directly infer that $\TS$ has polynomial decay with rate $1/d=1/3$ for all strictly  accretive $S\in L^\infty(\Omega_h,\C)$, we instead demonstrate assumption~\eqref{eq:polynomialboundforD} ourselves with $p(\lambda) \equiv 1$, which together with Proposition~\ref{prop:reductiontohzero} yields the improved polynomial decay rate $1/d=1/2$ for all strictly  accretive $S\in \domS $.

\begin{lemma}\label{lem:1dresidualbound}
    Let $\Omega_w$ be an interval in $\R,$ and let $\gamma$ be one of the boundary points of $\Omega_w.$ The mixed boundary wave equation generator $D_M$ on $(\Omega_w,\gamma)$, defined in \eqref{eq:defofDM}, satisfies the resolvent-type estimate 
    $$\|(\lambda i +D_M)w\|_2 \geq \|w\|_2 $$ for all $\lambda \in \R$ and all  $w=(w_1,w_2)\in \dom{D_M}$ with $\trg(w_1)=w_1\restrict{\gamma}=0.$
\end{lemma}
\begin{proof}
    After translating and rescaling, we may assume $\Omega_w= (0,1), \gamma = \{0\}.$ Let $\lambda\in \R,w\in \dom{D_M}$, and set $f=(f_1,f_2)^\top  \coloneq  (\lambda i +D_M)w$.
    A direct computation yields
    $$
        \dot{w}(x) = \begin{pmatrix}
            \dot{w}_1(x)\\\dot{w}_2(x)
        \end{pmatrix} = \begin{pmatrix}
            0 & -i\lambda \\ -i\lambda & 0
        \end{pmatrix}w(x)+\begin{pmatrix}
            f_2(x)\\f_1(x)
        \end{pmatrix}$$ for $x\in (0,1)$.

    In view of $w\in \dom{D_M}$, specifically $w_2(0)=-\trg_N(w_2)=0$ and $w_1(0)=\trg(w_1)=0,$
    we infer 
    $$w(x) = \int_0^x \exp(J(x-t))\begin{pmatrix}
        f_2(t)\\f_1(t)
    \end{pmatrix}\dx t,$$ where $J$ denotes the skew-Hermitian matrix 
    $$J=\begin{pmatrix}
            0 & -i\lambda \\ -i\lambda & 0
        \end{pmatrix}.$$
    Since $\exp(J(x-t))$ is unitary, we find 
    $$|w(x)| \leq \int_0^x \modulus{\begin{pmatrix}
        f_2(t)\\f_1(t)
    \end{pmatrix}} \dx t \leq \|f\|_{L^2((0,1))}~~\forall x\in (0,1),$$ and hence $\|w\|_{L^2((0,1))}\leq \|w\|_{L^\infty((0,1))} \leq \|f\|_{L^2((0,1))},$ which completes the proof. 
\end{proof}
\begin{remark}\label{rem:1doptimal}
   Consider the sequence $$w_k  \coloneq  \sin(\pi x  )\begin{pmatrix}
        \sin(k\pi x)\\
        i\cos(k\pi x)
    \end{pmatrix}, ~~\lambda_k = k\pi,~~k\in \Z$$ on $(\Omega_w,\gamma) = ((0,1), \{0\})$, which satisfies
    \begin{align*}
        w_k \in \dom{D_M}, ~~\trg((w_k)_1)=0,~~
        ((\lambda i +D_M)w_k)(x) = \pi \cos(\pi x)\begin{pmatrix}
            i\cos(\pi k x)\\
            \sin(\pi k x)
        \end{pmatrix}
    \end{align*} and hence $\pi \|w_k\|_2 = \|(\lambda i +D_M)w_k\|_2$ for all $k\in \Z.$
    From this, we see that the resolvent-type estimate in Lemma~\ref{lem:1dresidualbound} is optimal up to a constant factor.
\end{remark}
In view of Proposition~\ref{prop:reductiontohzero} and Theorem~\ref{thm:abstractresolventbounds}.(ii), Lemma~\ref{lem:1dresidualbound} directly implies the polynomial decay of the wave-heat semigroup $\TS$. 
We note that in one dimensions, $S\in L^\infty(\Omega_h,\C)$ is strictly  accretive if and only if  $\essinf_{x\in \Omega_h} \re(S(x))>0.$
\begin{corollary}[One-dimensional polynomial decay] \label{cor:1dpolynomialdecay}
Let $(\Omega,\Omega_w,\Omega_h)$ be a Lipschitz interface triple in $\R$ and let $S\in L^\infty(\Omega_h,\C)$ satisfy $$\essinf_{x\in \Omega_h} \re(S(x))>0.$$ Then the semigroup $\TS$ of the wave-heat system~\eqref{eq:1} 
    has polynomial decay with rate $1/d = 1/2$, i.e., 
\begin{align*}
        \forall k \in \N ~ \exists C_k>0: ~ \left\|\TS(t)x\right\|_{\bbH_1} \leq \frac{C_k}{(1+t)^{\frac{k}{2}}} \|x\|_{(\AS )^k} ~~ \forall t\geq0,  x\in \dom{(\AS )^k}.
\end{align*}
Additionally, $\TS$ does not possess polynomial decay of rate $1/d>2$ for all $S\in L^\infty(\Omega_h,\C)$, as was shown in \cite{1dimoptimalrate}.
\end{corollary}

\subsection{Geometric Control Condition}
\label{sec:GCC}

We next introduce an additional geometric assumption, which was shown to be closely related to the decay of wave-heat systems by Rauch, Zhang, Zuazua in \cite{zuazuapolynomial, zuazualongtime}.

\begin{definition}[Control in time]\label{def:GCC}
    Let $(\Omega,\Omega_w,\Omega_h)$ be a Lipschitz interface triple in $\R^n$ with interface $\gamma$. We say that $\Omega_h$ controls $\Omega_w$ in time if there exists a time $\tau>0$ such every wave equation solution $W$ with Dirichlet boundary conditions, i.e., every $$W \in C^0\left(\R_{\geq 0}, H^\Delta_0(\Omega)\right)\cap C^1\left(\R_{\geq 0}, H^1_0(\Omega)\right)\cap C^2\left(\R_{\geq 0}, L^2(\Omega)\right)$$ with
    \begin{align*}
        \partial_t^2 W(t,x) = \Delta W(t,x)~~ &\forall(t,x) \in (0,\infty)\times \Omega,\\
        W(0,\cdot) = w_0 \in H^\Delta_0(\Omega),~~ &\partial_t W(0,\cdot) = w_1 \in H^1_0(\Omega),
    \end{align*}
    satisfies the observability inequality 
    \begin{align}
        \|w_0\|_{H^1(\Omega)}^2 + \|w_1\|_{L^2(\Omega)}^2 \lesssim \int_0^\tau  \|\partial_t W(t,\cdot)\|^2_{L^2(\Omega_h)} \dx t. \label{eq:observintime}
    \end{align}
    
    Similarly, for admissible Lipschitz pairs $(\Omega_w,\gamma)$, we say that $\gamma \subset \dO_w$ controls $\Omega_w$ if there exists a Lipschitz interface triple $(\Omega,\Omega_w,\Omega_h)$ with interface $\gamma$ such that $\Omega_h$ controls $\Omega_w$ in time.
\end{definition}

The observability inequality~\eqref{eq:observintime} was shown to hold by Bardos, Lebeau, and Rauch  in \cite{rauchGCC} if $\Omega_w$ is smooth enough and $\Omega_h$ satisfies the Geometric Control Condition (GCC) in $\Omega$, i.e., every generalized geodesic in $\Omega$ traveling at unit speed, and reflecting off $\dO_w$ according to the laws of geometric optics, enters $\Omega_h$ in uniformly bounded time $\tau>0.$ See also Figure \ref{fig:GCC}.

We refer to \cite{DuyckaertsOptimal,zuazualongtime,zuazuapolynomial} for more detail and to \cite{BurqGCC} for weaker regularity assumptions on $\dO$.

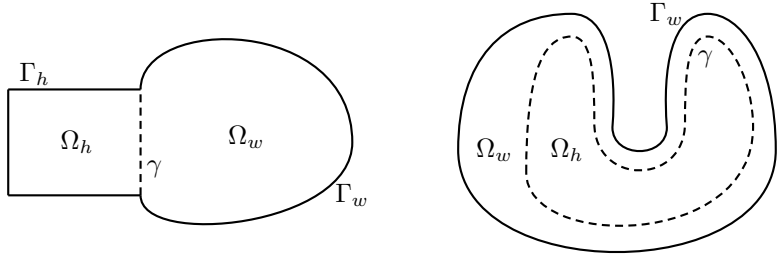
\begin{figure}[htbp]
    \centering
\begin{tikzpicture}[thick]
\begin{scope}[scale=1.4,shift={(0,0)}] 

\draw[densely dashed] (0,0) -- (0,1);

\draw (-1.25,0)--(0,0);
\draw (-1.25,1)--(0,1);
\draw (-1.25,1)--(-1.25,0);

\draw
(0,1)
.. controls (0,1.7) and (2,1.7) .. (2,0.5)
.. controls (2,-0.3) and (0,-0.5) .. (0,0);

\node at (-0.6,0.5) {$ \Omega_h$};
\node at (0.13,0.25) {$\gamma$};
\node at (1,0.55) {$\Omega_w$};
\node at (2,0) {$\Gamma_w$};
\node at (-1,1.15) {$\Gamma_h$};

\end{scope}
\begin{scope}[scale = 0.6, shift={(11,1)}]

\draw
(-4,0)
.. controls (-4,2) and (-3,3) .. (-1.5,3)
.. controls (-0.7,3) and (-0.5,1.8) .. (-0.6,0.5)
.. controls (-0.6,-0.2) and (0.6,-0.2) .. (0.6,0.5)
.. controls (0.5,1.8) and (0.7,3) .. (1.5,3)
.. controls (2.2,3) and (3,2) .. (3,0)
.. controls (3,-3) and (-4,-3) .. (-4,0)
-- cycle;

\draw[densely dashed]
(-2.5,-0.5)
.. controls (-2.5,1.5) and (-2.2,2.5) .. (-1.5,2.5)
.. controls (-1,2.5) and (-1,1.4) .. (-1,0.6)
.. controls (-1,-0.8) and (1,-0.8) .. (1,0.6)
.. controls (1,1.4) and (1,2.5) .. (1.5,2.5)
.. controls (2,2.5) and (2.5,1.5) .. (2.5,0.5)
.. controls (2.5,-2.2) and (-2.5,-2.2) .. (-2.5,-0.5)
-- cycle;

\node at (-3.2,0) {$\Omega_w$};
\node at (-1.6,0) {$\Omega_h$};
\node at (1.45,2) {$\gamma$};
\node at (0.6,3) {$\Gamma_w$};
\end{scope}
\end{tikzpicture}
\caption{Two Lipschitz interface triples, only the right one satisfies the Geometric Control Condition, even though $\Omega_h$ is surrounded by $\Omega_w$, i.e., $\Gamma_h = \varnothing$.}
    \label{fig:GCC}
\end{figure}

\begin{remark}\label{rem:lions}
A common case in which $\Omega_h$ controls $\Omega_w$ in time is when  Lion's multiplier condition \cite{lionsmutlipliercondition} is fulfilled, which, for smooth enough $\Gamma_w,$ states that there exists an $x_0 \in \R^n$ such that $$ (x-x_0)\cdot \nu_w(x) \leq 0 ~~\forall x\in \Gamma_w,$$
see also Figure~\ref{fig:lionsmultiplercond}.
Regarding estimates on the resulting observability time $\tau$, we refer to Remark 3.4 in \cite{lionsmutlipliercondition}.

We note that in the case $\Gamma_w = \varnothing,$ i.e., in the case where $\Omega_h$ completely surrounds $\Omega_w$, Lion's multiplier condition and hence the observability estimate $\eqref{eq:observintime}$ is automatically satisfied.
\end{remark}
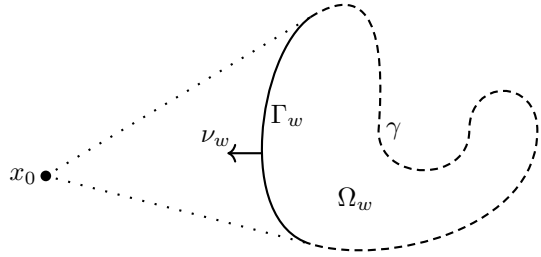
\begin{figure}[htbp]
    \centering
\begin{tikzpicture}[thick]
\begin{scope}[scale=1,shift={(0,0)}] 
\draw[densely dashed]
(-1.5,2)
.. controls (-0.7,2.5) and (-0.5,1.8) .. (-0.6,0.5)
.. controls (-0.6,-0.2) and (0.6,-0.2) .. (0.6,0.5)
.. controls (0.6,1.2) and (1.5,1.2) .. (1.5,0.5)
.. controls (1.5, -0.8) and (-0.5,-1.3) .. (-1.5,-1);
\draw
(-1.5,2)
.. controls (-2.2,1.5) and (-2.5,-0.7) .. (-1.5,-1);

\draw[->, shorten >= 2mm] (-2.14,0.2) -- (-2.8,0.2);
\node at (-2.75,0.4) {$ \nu_w$};

\node at (-1.8,0.7) {$ \Gamma_w$};

\node at (-0.9,-0.4) {$ \Omega_w$};
\node at (-0.4, 0.5) {$ \gamma$};
\node (A) at (-5,-0.1) {};
\draw[loosely dotted]
(A)-- (-1.5,2);
\draw[loosely dotted]
(A)-- (-1.5,-1);
\node at (-5.3,-0.1) {$x_0$};
\fill (A) circle (2pt);
\end{scope}
\end{tikzpicture}
    \label{fig:placeholder}
    
\caption{$\Gamma_w$ satisfies Lion's multipler condition.}
    \label{fig:lionsmultiplercond}
\end{figure}
In the following, we show that the observability estimate~\eqref{eq:observintime} and hence the Geometric Control Condition imply a relatively fast polynomial decay rate  for the semigroup $\TS$, i.e., there is some $d>0$ such that
\begin{align*}
        \forall k \in \N ~ \exists C_k>0: ~ \left\|\TS(t)x\right\|_{\bbH_1} \leq \frac{C_k}{(1+t)^{\frac{k}{d}}} \|x\|_{(\AS)^k} ~~ \forall t\geq0,  x\in \dom{(\AS)^k}.
\end{align*}
This was first observed in the case $S\equiv 1$ by Zhang, Zuazua in \cite{zuazualongtime} with $1/d=1/6$, and then later improved by Duyckearts \cite{DuyckaertsOptimal} to any rate $1/d <1.$
Our methods yield a slightly slower rate of $1/p = 1/2$ but allow for non-trivial,  strictly  accretive heat coefficients $S\in \domS$.

The proof of the following result roughly follows a Hautus-test-type argument, where we deduce a resolvent-type bound (see \eqref{eq:residualonOmegaDD}) from an observability in time estimate (see \eqref{eq:hautustestsimilar}). The observability in time estimate is inferred via an auxiliary wave equation.
First, however, we transform the claimed resolvent-type estimate involving $D_M$ on $\Omega_w$ into a resolvent-type estimate involving $D_D$ (defined in \eqref{eq:defofDD}) on $\Omega$. 
\begin{proposition} \label{prop:GCCimpliesD}
    Let $(\Omega_w,\gamma)$ be an admissible Lipschitz pair in $\R^n$ with $\gamma$ controlling $\Omega_w$ in time.
    Then the resolvent-type bound 
    \begin{align*}
    \|(\lambda i+D_M)w\|_2 \gtrsim \|w\|_2, ~~|\lambda|\gg1,\lambda\in \R,
\end{align*} holds for all $w=(w_1,w_2)  \in \dom{D_M}$ with $\trg(w_1)= 0.$
\end{proposition}
\begin{proof}
    
    Firstly, let 
    $(\Omega,\Omega_w,\Omega_h)$ be a Lipschitz interface triple with interface $\gamma$ such that $\Omega_h$ controls $\Omega_w$, and let $D_D$ denote the Dirichlet wave equation generator on $\Omega$, not $\Omega_w$, namely
    \begin{align*}
    D_D=\begin{pmatrix}
        0 & \di \\ \gr & 0
    \end{pmatrix} :H^1_{0}(\Omega) \times H^\di(\Omega) \subset L^2(\Omega)^{1+n}\to L^2(\Omega)^{1+n}.
    \end{align*} 
    In view of the proof of Lemma~\ref{lem:contractive}, $D_D$ is skew-adjoint.    We note that $D_M$ still acts on $L^2(\Omega_w)$, not $L^2(\Omega)$.
    
    Let $w = (w_1,w_2)  \in \dom{D_M}$ with $\trg(w_1) = 0$ be arbitrary, the extension of $w$ by zero to $\Omega$, which we denote by $w' = (w'_1,w'_2) $, is then an element of $\dom{D_D},$ see Proposition~\ref{prop:patching}.
    Due to 
    $$\|w'\|_{L^2(\Omega)} = \|w\|_{L^2(\Omega_w)},~~\|(\lambda i +D_D)w'\|_{L^2(\Omega)} = \|(\lambda i +D_M)w\|_{L^2(\Omega_w)},$$ it suffices to show 
    \begin{align}\label{eq:residualonOmegaDD}
        \|(\lambda i +D_D)v\|_{L^2(\Omega)} \gtrsim \|v\|_{L^2(\Omega)}, ~~ |\lambda|\bb 1, \lambda\in \R
    \end{align}
    for all $v \in \dom{D_D}$ with $v\restrict{\Omega_h} = 0.$

    Secondly, 
    for $v \in \dom{D_D}$ with $v\restrict{\Omega_h} = 0$, we define
    $$w_\gr  \coloneq  \begin{pmatrix} v_1 \\ (1-\bbP_{\gamma})v_2\end{pmatrix}, ~~ w_0  \coloneq  \begin{pmatrix}
        0 \\ \bbP_{\gamma} v_2
    \end{pmatrix},$$
    where $\bbP$ denotes the Helmholtz projection in $L^2(\Omega)$ with kernel $\gr H^1_{0}(\Omega)$ and range $\{v \in H^\di(\Omega)~|~ \di(v) = 0\}$,
    see Theorem 4.2 in \cite{hodgeonforms}.
    
    As in the proof of Proposition~\ref{prop:interfacextension}, this ensures $w_\gr \in H^1_0(\Omega)\times \gr H^\Delta_0(\Omega),$ in particular, there exists a potential $p\in H^\Delta_0(\Omega)$ with $\gr p = (w_\gr)_2.$
    
    Let $T_D: [0,\infty) \times L^2(\Omega)^{1+n}\to L^2(\Omega)^{1+n}$ be the unitary semigroup generated by $D_D.$
    With this, we find that 
    $$W(t,x)  \coloneq  p(x)+\int_0^t (T_D(s)w_\gr )_1(x) \dx s$$ is an element of 
    $$C^0\left(\R_{\geq 0}, H^\Delta_0(\Omega)\right)\cap C^1\left(\R_{\geq 0}, H^1_0(\Omega)\right)\cap C^2\left(\R_{\geq 0}, L^2(\Omega)\right)$$ 
    and a solution to the wave equation
     \begin{align*}
        \partial_t^2 W(t,x) = \Delta W(t,x)~~ &\forall(t,x) \in (0,\infty)\times \Omega,\\
        W(0,\cdot) = p,~~& \partial_t W(0,\cdot) = v_1.
    \end{align*}
    
    Indeed,  we observe
    \begin{align*}
        \partial_t W(t,\cdot) = (T_D(t)w_\gr)_1 \in H^1_0(\Omega)
    \end{align*}
    since $T_D$ maps $\dom{D_D}$ into $\dom{D_D}.$
    Moreover, \begin{align*}
        \gr W(t,\cdot) &= \gr p +\int_0^t  \gr (T_D(s)w_\gr)_1 \dx s
        = (w_\gr)_2 +\int_0^t (D_D T_D(s)w_\gr)_2 \dx s
        \\&= (w_\gr)_2 +\left(\int_0^t \partial_t T_D(s)w_\gr \dx s\right)_2
        = \left( T_D(t)w_\gr \right)_2,
    \end{align*}
    which implies
    \begin{align*}
        \partial_t^2 W(t,\cdot) = \partial_t (T_D(t) w_\gr)_1 = \di ( T_D(t) w_\gr )_2 = \Delta W(t,\cdot) \in L^2(\Omega).
    \end{align*}
    This also shows $$W \in C^0\left(\R_{\geq 0}, H^\Delta_0(\Omega)\right)\cap C^1\left(\R_{\geq 0}, H^1_0(\Omega)\right)\cap C^2\left(\R_{\geq 0}, L^2(\Omega)\right).$$ 
    The initial data $W(0,\cdot) = p,~\partial_t W(0,\cdot) = v_1$ is clear. 

    Applying the assumption that $\Omega_h$ controls $\Omega_w$ in time to $W$, we find
    \begin{align}
        \|w_\gr\|^2_{L^2(\Omega)}&=\|(w_\gr)_2\|_{L^2(\Omega)}^2 + \|(w_\gr)_1\|_{L^2(\Omega)}^2 \leq \|p\|_{H^1(\Omega)}^2 + \|w_1'\|_{L^2(\Omega)}^2 \nonumber
        \\&\lesssim \int_0^\tau  \|\partial_t W(t,\cdot)\|^2_{L^2(\Omega_h)} \dx t
        \lesssim \int_0^\tau  \|(T_D(t) w_\gr)_1\|^2_{L^2(\Omega_h)} \dx t. \label{eq:hautustestsimilar}
    \end{align} 
    Re-adding the orthogonal term $w_0$ to $w_\gr$, we infer
    \begin{align*}
        \|v\|_{L^2(\Omega)}^2 &= \|w_\gr\|^2_{L^2(\Omega)} +\|w_0\|_{L^2(\Omega)}^2 
        \\&  \lesssim \int_0^\tau  \|(T_D(t) w_\gr)_1\|^2_{L^2(\Omega_h)} \dx t+\|w_0\|_{L^2(\Omega)}^2
        \\&  \leq \int_0^\tau  \|T_D(t) w_\gr\|^2_{L^2(\Omega_h)} \dx t+\|w_0\|_{L^2(\Omega)}^2
        \\&  \leq \int_0^\tau  \|T_D(t) v\|^2_{L^2(\Omega_h)}+\|T_D(t) w_0\|^2_{L^2(\Omega_h)} \dx t+\|w_0\|_{L^2(\Omega)}^2
        \\&  \lesssim \int_0^\tau  \|T_D(t) v\|^2_{L^2(\Omega_h)} \dx t+\|w_0\|_{L^2(\Omega)}^2.
    \end{align*} 
    If we now temporarily assume the estimate 
    \begin{align} \label{eq:temp}
        \int_0^\tau  \|T_D(t) v\|^2_{L^2(\Omega_h)} \dx t\lesssim \|(\lambda i +D_D)v\|^2_2,~~ |\lambda|\bb 1, \lambda\in \R,
    \end{align} then 
     \begin{align*}
        \|v\|_{L^2(\Omega)}^2 &\lesssim \|(\lambda i +D_D)v\|_2^2+\|w_0\|_{L^2(\Omega)}^2 
        \\&= \|(\lambda i +D_D)w_\gr \|_2^2+(1+|\lambda|)\|w_0\|_{L^2(\Omega)}^2 
        \\&\lesssim \|(\lambda i +D_D)v \|_2^2,
    \end{align*} which we previously observed was enough to conclude.
    Therefore, it remains to demonstrate estimate~\eqref{eq:temp} for all $v \in \dom{D_D}$ with $v\restrict{\Omega_h}=0.$ 

    Let $r$ denote the term $(\lambda i +D_D)v$ and let $S(t)v \coloneq  \exp(i\lambda t)v$ denote the stationary wave ansatz, which is identically zero on $\Omega_h$. The error $E(t)  \coloneq  (T_D(t)-S(t))v$ then satisfies 
    \begin{align*}
        \partial_t E(t) &=  -i \lambda S(t)v+ D_D T_D(t)v\\&=  D_D S(t)v+ D T_D(t)v-(i\lambda +D_D)S(t)v 
        \\&= D_D E(t) - \exp(i\lambda t)r.
    \end{align*}
    To estimate the size of the error $E(t)$, we observe $E(0) = 0$ and
    \begin{align*}
       2\|E(t)\|_{L^2(\Omega)}\partial_t\|E(t)\|_{L^2(\Omega)}
        &=\partial_t\|E(t)\|_{L^2(\Omega)}^2 
        \\&= 2 \re \sca{\partial_t E(t)}{E(t)}_{L^2(\Omega)} 
        \\&=2 \re \sca{D E(t)-\exp(i\lambda t)r}{E(t)}_{L^2(\Omega)}
        \\& \lesssim\|r\|_{L^2(\Omega)}\|E(t)\|_{L^2(\Omega)} ,
    \end{align*} where the last step follows from the skew-adjointness of $D.$
    Integrating this yields $\|E(t)\|_{L^2(\Omega)} \lesssim t\|r\|_{L^2(\Omega)}$.
    
    Lastly, we find
    \begin{align*} 
        \int_0^\tau  \|T_D(t) v\|^2_{L^2(\Omega_h)} \dx t &= \int_0^\tau  \|E(t)\|^2_{L^2(\Omega_h)} \dx t
        \lesssim \int_0^\tau  t^2 \|r\|^2_{L^2(\Omega)} \dx t
        \\&\approx \tau^3\|r\|^2_{L^2(\Omega)}
        = \tau^3\|(\lambda i +D_D)v\|^2_{L^2(\Omega)}.
    \end{align*} 
    As $\tau$ depends only on the domain and can hence be considered constant,
 we have indeed demonstrated estimate~\eqref{eq:temp} and 
    thereby completed the proof.\qedhere
    
\end{proof}

Keeping in mind Proposition~\ref{prop:reductiontohzero} and Theorem~\ref{thm:abstractresolventbounds}.(ii), Proposition~\ref{prop:GCCimpliesD} directly implies the sought after polynomial decay of the wave-heat semigroup $\TS$.
\begin{corollary}\label{cor:GGCpolynomrate}
    Let $(\Omega,\Omega_w,\Omega_h)$ be a Lipschitz interface triple in $\R^n$ with $\Omega_h$ controlling $\Omega_w$ in time, and let $S \in \domS$ be strictly  accretive.
    Then the semigroup $\TS$ of the wave-heat system~\eqref{eq:1} 
    has polynomial decay with rate $1/d = 1/2$, i.e., 
\begin{align*}
        \forall k \in \N ~ \exists C_k>0: ~ \left\|\TS(t)x\right\|_{\bbH_1} \leq \frac{C_k}{(1+t)^{\frac{k}{2}}} \|x\|_{(\AS )^k} ~~ \forall t\geq0,  x\in \dom{(\AS )^k}.
\end{align*}
\end{corollary}

We note that this result closely mirrors Theorem 2.3 in \cite{dampedwaveandhautus}, where a generalized version of a wave equation with interior damping was considered.
Furthermore, Corollary~\ref{cor:GGCpolynomrate} can also be used to find a second proof of Corollary~\ref{cor:1dpolynomialdecay}, by demonstrating that every $1$-dimensional domain is controlled in time by each non-empty boundary part.

Lastly, Figure~\ref{fig:implications} presents an informal overview of the different results relating to polynomial decay rates on $\TS$ that were shown in this paper.

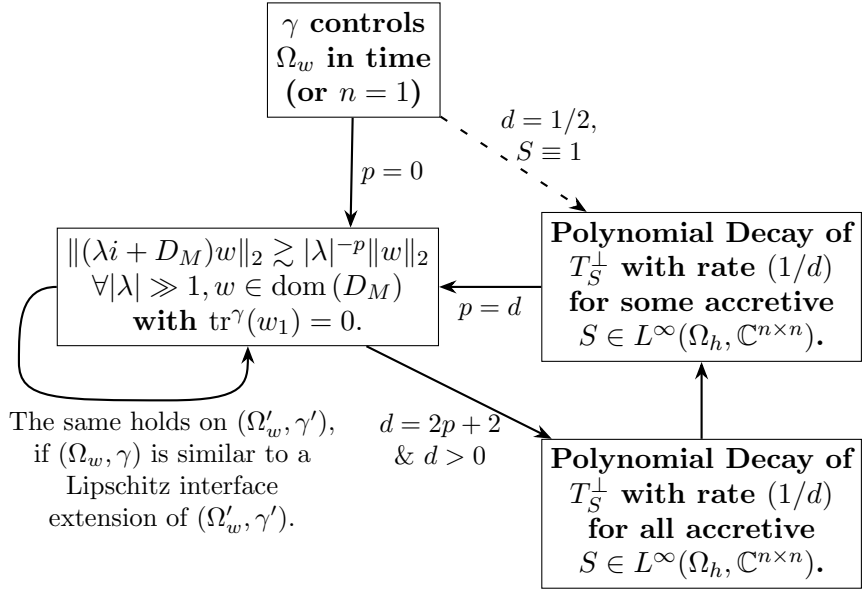
\begin{figure}
    \centering
\begin{tikzpicture}[
    node distance=2cm,
    block/.style={ draw,minimum size=1cm},
    align=center,
    arrow/.style={-{Stealth}, thick},
    doublearrow/.style={<->, {Stealth}-{Stealth}, thick}
    labelstyle/.style={
        draw=none,    
        fill=none,    
        font=\tiny, 
        color=black
    },
]

    \node[block] (E) at (0, -3) {$\|(\lambda i+D_M)w\|_2 \gtrsim |\lambda|^{-p}\|w\|_2$\\ $\forall |\lambda|\gg1, w\in \dom{D_M}$\\ with $\trg(w_1)= 0.$};
    \node[block] (F) at (6.5, -3) {Polynomial Decay of\\$\TS$ with rate $(1/d)$ for\\some strictly accretive\\ $S\in \domS $.};
    \node[block] (G) at (6.5, -6) {Polynomial Decay of\\$\TS$ with rate $(1/d)$ for\\all strictly accretive\\ $S\in \domS $.};
    \node[block] (GCC) at (1.4, 0) {$\gamma$ controls\\$\Omega_w$ in time\\(or $n=1$)};

    \draw[arrow] (E.180) 
        .. controls ++(-0.5,0) and ++(0,0.5) .. ++(-0.5,-1) 
        .. controls ++(0,-0.5) and ++(-1.5,0) .. ++(2,-0.5) 
        node[above = 0pt] at (-1,-4.5) {\ref{thm:domdependence}}
        node[below] at (-1,-4.5) 
        {The same holds on $(\Omega'_w,\gamma'),$\\ if $(\Omega_w,\gamma)$ is similar to a \\Lipschitz interface\\extension of $(\Omega'_w,\gamma')$.}
        .. controls ++(1,0)  and ++(0,-0.5) .. (E.270);

    \draw [arrow, ] (F) to node[below,yshift=0mm] {$p = d$}
    node[above] {\ref{thm:abstractresolventbounds},~\ref{rem:afterhzeroreductiob}}
    (E);
    \draw [arrow, ] (G) to (F);
    \draw [arrow, ] (E) to node[below, xshift =-2.5mm, yshift=-1.5mm] {$d =  2p+2$\\\& $d>0$}
    node[above, xshift =4mm, yshift=-1mm] {\ref{thm:abstractresolventbounds},~\ref{prop:reductiontohzero}}
    (G);
    \draw[arrow, loosely dashed] (GCC) to node[right,yshift=4mm,xshift = -3mm] {$d=1/2$, \\$S\equiv 1$} (F);
    \draw[arrow] (GCC) to node[right]{$p=0$} node[left]{\ref{cor:1dpolynomialdecay},~\ref{cor:GGCpolynomrate}} (E.28);

\end{tikzpicture}
    \caption{An informal overview of all the results relating to polynomial decay rates on $\TS$. 
    The dashed arrow was shown in \protect\cite{1dimoptimalrate,DuyckaertsOptimal,zuazua1d}, while the solid arrows correspond to the results established in this paper. The solid arrows also contain references to where their result is shown.}
    \label{fig:implications}
    
\end{figure}

\subsection{Outlook.}

In future work, we aim to find observability-type conditions which are necessary and sufficient for the sought after resolvent-type estimate~\eqref{eq:polynomialboundforD} to hold.

By disproving these observability conditions in certain geometric regimes, we are able to find a first Lipschitz interface triple on which $\TS$ does not possess polynomial decay rates. Even further, this allows us to show that the logarithmic decay
rates obtained in Theorem~\ref{thm:logdecay} are optimal, i.e., there are Lipschitz interface triples on which $\TS$ does not have logarithmic decay with any rate $1/d>1$.

Moreover, we use methods from harmonic analysis and lattice geometry to demonstrate these sufficient observability estimates in different geometric regimes, allowing us to find explicit polynomial decay rates on $\TS$ in multiple new cases.
In particular, we aim to generalize the polynomial decay on $T_{1}^\perp$ obtained in \cite{rectangulardomains}, where the Lipschitz interface triple was given by two adjacent, two-dimensional squares.

\appendix

\section{More on traces and extension operators} \label{sec:appendix}

This section is about three additional results related to traces and extension operators, whose proofs were postponed.

First, we present the proof of
Proposition~\ref{prop:patching}, which we repeat here as Proposition~\ref{prop:patchinginappendix} for convenience.

\begin{proposition}\label{prop:patchinginappendix}
    Let $(\Omega,\Omega_w,\Omega_h)$ be a Lipschitz interface triple with interface $\gamma$. Let $\pi\in L^2(\Omega)$ and set $w \coloneq  \pi \restrict{\Omega_w}, h  \coloneq  \pi \restrict{\Omega_h}.$
    \begin{enumerate}
        \item Assume $w\in H^1(\Omega_w)$ and $ h\in H^1(\Omega_h)$. The functions $w,h$ have matching boundary trace, i.e., $$\trg(w) =\trg(h) \tn{ in }\heh(\gamma),$$ if and only if $\pi$ is an element of $ H^1(\Omega)$.
        \item Assume $w\in H^\di(\Omega_w) $ and $ h\in H^\di(\Omega_h)$. The functions $w,h$ have matching normal trace, i.e., $$\trg_N(w) = -\trg_N(h) \tn{ in } \hnehz(\gamma),$$ if and only if $\pi$ is an element of $ H^\di(\Omega)$.
        
        \item Assume $w\in H^\Delta(\Omega_w)$ and $h\in H^\Delta(\Omega_h)$. The functions $w,h$ have matching boundary trace and normal derivative, i.e., $$\qquad \trg(w) =\trg(h)\tn{ in } \heh(\gamma)  \quad \tn{and}\quad \trg_N(\gr w) = -\trg_N(\gr h) \tn{ in } \hnehz(\gamma),$$ if and only if $\pi$ is an element of $ H^\Delta(\Omega)$.
    \end{enumerate}
\end{proposition}
\begin{proof}

\begin{enumerate}[wide, labelwidth=0pt, labelindent=0pt] 
  
\item   We first note the divergence theorem for Lipschitz domains.
On any bounded Lipschitz domain $\Omega$ with unit outer normal vector $\nu$, one has
\begin{align}
    \io \di(f ) \dx x &=  \int_{\dO}  \tr(f) \cdot \nu \dx \sigma, \label{eq:divtheorems}\\
   \io \partial_{x_j} f \dx x &=  \int_{\dO}  \tr(f) \cdot \nu_{j} \dx \sigma \nonumber
\end{align}
for all  $\Phi \in H^1(\Omega)^n$ and $j\in \{1,\dots,n\}$.  
This follows from identity \eqref{eq:normaltraceidentitybydensity} by substituting $\psi = 1$ and $\tr_N(f) = \tr(f)\cdot \nu$.

Assume now that $w \in H^1(\Omega_w), h \in H^1(\Omega_h)$ have matching trace $\trg(w) = \trg(h)$. We have to show $\pi \in  H^1(\Omega)$.  Both $\pi$ and its hypothetical gradient $$q(x) \coloneq  \left\{
\begin{array}{ll}
(\gr w)(x), & x \in \Omega_w \\
(\gr h)(x), & x \in \Omega_h  \\
\end{array}
\right.$$ are clearly $L^2$-integrable on $\Omega$. Therefore, it suffices to confirm $q = \gr \pi$ in the weak sense, i.e.,  $\io  \pi \partial_{x_j} \varphi + q_j \varphi \dx x = 0$ for all $\varphi\in C^\infty_c(\Omega)$, $1\leq j\leq n$. 

Let $\nu_w,\nu_h$ denote the unit outer normal vectors on $\partial \Omega_w,\partial \Omega_h$.
Keeping in mind $\varphi\restrict{\Gamma_w}=0=\varphi\restrict{\Gamma_h}$ as well as the second identity in \eqref{eq:divtheorems}, we directly verify
\begin{align*}
    \io  \pi \partial_{x_j} \varphi + q_j \varphi \dx x &=
    \iow   \partial_{x_j} (w \varphi) \dx x + \ioh \partial_{x_j} (h \varphi)  \dx x \\&=  \int_{\partial \Omega_w } \tr(w)  \varphi \cdot(\nu_w)_j \dx \sigma +\int_{\partial \Omega_h}  \tr(h) \varphi \cdot(\nu_h)_j  \dx \sigma\\&= \ig\trg( w)  \varphi \cdot(\nu_w)_j -\trg(h)  \varphi \cdot(\nu_w)_j  \dx \sigma=0.
\end{align*}

Conversely, assume $\pi \in H^1(\Omega)$. Let $(\psi_k)_{k\in \bbN} \in C^\infty(\overline{\Omega})$ be an approximating sequence of $\pi$ in $H^1(\Omega)$, which exists by Theorem 3.29 of \cite{McLean}. Restricting this sequence to $\Omega_w,\Omega_h$ then yields an approximating sequence of $w,h$ in $ H^1(\Omega_w), h\in H^1(\Omega_h)$. By the continuity of the trace operator, we have 
\begin{align*}
    \lim_{k\to \infty} \|\trg(w)-\psi_k\restrict{\gamma}\|_{H^{1/2}(\gamma)} &\leq \lim_{k\to \infty} \|\tr(w)-\psi_k\restrict{\partial \Omega_w}\|_{H^{1/2}(\partial \Omega_w)}\\&\lesssim\lim_{k\to \infty} \|w-\psi_k\restrict{ \Omega_w}\|_{H^{1}(\Omega_w)} = 0,
\end{align*} and analogously $\lim_{k\to \infty} \|\trg(h)-\psi_k\restrict{\gamma}\|_{H^{1/2}(\gamma)} = 0$, proving $\trg(w)= \trg(h)$ in $H^{1/2}(\gamma)$.\\

    \item Assume $w\in H^\di(\Omega_w), h\in H^\di(\Omega_h)$, and $\trg_N(w) = - \trg_N(h)$. We have to show $\pi \in H^\di(\Omega)$. Both $\pi$ and its hypothetical divergence $$q(x) \coloneq  \left\{
\begin{array}{ll}
(\di (w))(x), & x \in \Omega_w \\
(\di (h))(x), & x \in \Omega_h  \\
\end{array}
\right.$$ are clearly $L^2$-integrable on $\Omega$. Therefore, it suffices to confirm $\di (\pi) = q$ in the weak sense, i.e., $$\io  q \varphi + \pi \cdot \gr \varphi \dx x= 0 ~~\forall\varphi\in C^\infty_c(\Omega).$$ Using identity~\eqref{eq:normaltraceidentity}, which requires $\varphi\restrict{\Gamma_w}= 0 = \varphi\restrict{\Gamma_h}$, we confirm 
    \begin{align*}
        \io  q \varphi + \pi \cdot \gr \varphi \dx x&=  \iow  \di(w) \varphi + w \cdot \gr \varphi\dx x+\ioh \di(h) \varphi + h \cdot \gr \varphi\dx x \\&= \duality{\trg_N(w)}{\varphi\restrict{\gamma}}+\duality{\trg_N(h)}{\varphi\restrict{\gamma}} = 0.
    \end{align*}
    Conversely, assume $w\in H^\di(\Omega_w), h\in H^\di(\Omega_h)$ and $\pi \in H^\di(\Omega)$. We have to show $\trg_N(w) = - \trg_N(h)$ in $\hnehz(\gamma).$ 
    For this, let $v \in \hehz(\gamma)$ be arbitrary, and let $ \psi \in H^1_0(\Omega)$ be such that $v = \trg(\psi \restrict{\Omega_w})= \trg(\psi \restrict{\Omega_h})$. The existence of $\psi$ follows from Statement (i). We indeed find
    \begin{align*}
       &\hspace{3ex} \duality{\trg_N(w)}{v} + \duality{\trg_N(h)}{v} \\&= \iow  \di(w) \psi+ w \cdot \gr \psi\dx x+\ioh \di(h) \psi+ h \cdot \gr \psi\dx x  \\& = \io \di(\pi \psi) \dx x = \duality{\tr_N(\pi)}{\tr(\psi)}= 0.
    \end{align*}
    
    \item We observe that a function $f$ is an element of $H^\Delta(\Omega)$ if and only if $f$ is an element of $H^1(\Omega)$ and $\gr f$ is an element of $H^\di(\Omega)$. Statement (iii) hence follows by combining Statement (ii) applied to $\gr f$ and Statement (i) applied to $f$.\qedhere
    \end{enumerate}
\end{proof}

Secondly, we show that similarly to the restricted (Dirichlet) trace operator
$\trG: H^1(\Omega)\to \heh(\Gamma),$ the restricted normal trace operator $\trG_N$ is also bounded and surjective.

 \begin{lemma} \label{lem:normaltracesurj} Let $\Gamma$ be a relatively open subset of the boundary $\dO$ of some
bounded Lipschitz domain $\Omega \subset \R^n.$
     The restricted normal trace operator $$\trG_N:H^\di(\Omega)\to \hnehz(\Gamma),$$ defined in Definition \ref{def:restrictedtraceoperators}.(iii), is well-defined, bounded and surjective.
 \end{lemma}
 
 \begin{proof}
    We recall that for each element $g$ of $\hehz(\Gamma)$, $\Tilde{g}$ denotes its extension by zero to $\heh(\dO)$, which by Definition \ref{def:restrictedtraceoperators}(ii) exists and satisfies $\|g\|_{\hehz(\Gamma)}\approx \|\Tilde{g}\|_{\heh(\dO)}$.
    
    For the well-definedness and boundedness, we use the boundedness of the non-restricted normal trace operator $\tr_N$ to estimate
     \begin{align*}
         \left\|\trG_N (f)\right\|_{\hnehz(\Gamma)} &= \sup_{\substack{g\in \hehz(\Gamma)\\\|g\|_{\hehz(\Gamma)} = 1}} \left|\duality{\tr_N(f)}{\Tilde{g}}\right|
         \approx    \sup_{\substack{g\in \hehz(\Gamma)\\\|\Tilde{g}\|_{\heh(\dO)} = 1}} \left|\duality{\tr_N(f)}{\Tilde{g}}\right|
         \\ & \leq \sup_{\substack{h \in \heh(\dO)\\\|h\|_{\heh(\dO)} =1}} \left|\duality{\tr_N(f)}{h}\right| = \|\tr_N(f)\|_{\hneh(\dO)} \lesssim \|f\|_{H^\di(\Omega)}.
     \end{align*}
     To infer the surjectivity of $\trG_N$,
     denote by $$\Tilde{H}^{1/2}_0(\Gamma)  \coloneq  \{g\in H^{1/2}(\dO)~|~ g = 0 \tn{ on } \dO \backslash\Gamma\}$$ the image of $\hehz(\Gamma)$ under extension by zero.
     We note that every linear functional $L\in\hnehz(\Gamma)$ can be extended to a linear functional $\Tilde{L} \in \hneh(\dO)$ by setting $$ \duality{\Tilde{L}}{h+h^\perp}  \coloneq  \duality{L}{h\restrict{\Gamma}}, ~~h\in \Tilde{H}^{1/2}_0(\Gamma), h^\perp\in \Tilde{H}^{1/2}_0(\Gamma)^\perp.$$
     By the surjectivity of $\tr_N:H^\di(\Omega)\to \hneh(\dO)$, there exists an element $f\in H^\di(\Omega)$ with $\tr_N(f) = \Tilde{L}$.
     The surjectivity of $\trG_N$ then follows from $\trG_N(f)=L$.
     Indeed, 
     $$\duality{\trG_N(f)}{g}=\duality{\tr_N(f)}{\Tilde{g}} = \duality{\Tilde{L}}{\Tilde{g}} = \duality{L}{g}$$ for all $g\in \hehz(\Gamma)$.
 \end{proof}

The following result was used in Lemma \ref{lem:smallextensionoperator} as part of the larger goal of decomposing $\dom{\AS}$ non-orthogonally into a part that is independent of $\Omega_h$, and a part whose $L^2$-mass is supported mostly on $\Omega_h$, where strong dissipation occurs.

\begin{proposition}[$L^2$-bounded $H^\di$-extension operator]
    \label{prop:hdivextensionop}
    
    Consider a bounded Lipschitz domain $\Omega\subset \R^n$. There then exists an extension operator $$E^\di_{\Omega}: H^\di(\Omega)\to H^\di(\R^n)$$ satisfying 
    $$\left(E^\di_{\Omega} v\right) \restrict{\Omega} = v ~~\tn{and}~~\left\|E^\di_{\Omega} v\right\|_{H^\di(\R^n)} \lesssim \|v\|_{H^\di(\Omega)}.$$ 
    This operator can be assumed to be $L^2$-bounded, i.e., to also satisfy $$\left\|E^\di_{\Omega} v\right\|_{L^2(\R^n)} \lesssim \|v\|_{L^2(\Omega)}.$$
\end{proposition}

    The existence of an $H^\di$-extension operator is well-known, see, e.g., \cite{bedivan1996extension} and Theorem 3.1 in \cite{gopalakrishnan2012partial}. That such an operator can be constructed to be $L^2$-bounded is rarely explicitly stated.
    We thus present a strategy of how such an operator can be constructed.
    \begin{proof}[Sketch of proof]
    \begin{enumerate}[Step 1.]
    \item If $n=1$, the statement is trivial, e.g.,  one can reflect $v$ along $\dO$ and multiply it with a cut-off function.
    \item   For $n\geq 2$, we first note that every (strongly) Lipschitz domain in the sense of Definition \ref{def:lipschitz}.(i) is also weakly Lipschitz, i.e., for every $x\in \dO$, there is an open neighborhood $U_x\subset \R^n$ around $x$ and a bi-Lipschitz map 
    $\Phi_x:U_x\to (-1,1)^n$ with 
    \begin{align*}
       \Phi_x(\Omega\cap U_x) = (-1,1)^{n-1}\times (-1,0).
    \end{align*} For more information, we refer to Chapter 3 of \cite{hodgeonforms}.
    
        \item  Using the relative compactness of $\Omega$, cover $\partial \Omega$ by a finite number of open Lipschitz patches $\{U_j\}_{j=1}^N$ so that each admits a bi-Lipschitz homeomorphism $\Phi_j: U_j \to (-1,1)^n$ with 
        $$\Phi_j(\Omega\cap U_j) = (-1,1)^{n-1}\times (-1,0).$$  
        We may assume each $\Phi_j$ to be orientation-preserving, i.e., $\det( D\Phi_j )>0$ almost everywhere, by potentially post-composing $\Phi_j$ with the reflection
        $x\mapsto  (-x_1,x_2,\dots,x_n).$
        
        Additionally, choose a domain $U_0$ with
        $$\Omega\backslash \bigcup_{j=1}^N U_j \subset U_0 \subset \overline{ U_0} \subset \Omega.$$
        \item 
        Let $v \in H^\di(\Omega)$ be arbitrary and let $\{\chi_j\}_{j=0}^N$ be a smooth partition of unity on $\Omega$ with $\supp\chi_j \subset U_j, j = 0,\dots,N.$ Restrict $v$ to $U_j$ by setting $v_j \coloneq  v \chi_j\restrict{U_j\cap \Omega} \in H^\di(U_j\cap \Omega)$.
        \item 
        For $ j = 1,...,N$, construct the Piola transforms $$~~~\hat{v}_j  \coloneq  \det (D \Phi_j\circ \Phi_j\inv)\left(D \Phi_j\circ \Phi_j\inv\right)\inv v_j\circ \Phi_j\inv,$$ 
        which are known to satisfy $\hat{v}_j \in H^\di((-1,1)^{n-1}\times(-1,0))$ with 
        \begin{align*}
            &(\di (\hat{v}_j) )(x) =  \det (D \Phi_j(\Phi_j\inv x)) \cdot (\di (v_j))(\Phi_j\inv x),\\ &\|\hat{v}_j\|_{L^2((-1,1)^{n-1}\times(-1,0))} \approx \|v_j\|_{L^2(U_j\cap \Omega)}, \\&\|\di (\hat{v}_j)\|_{L^2((-1,1)^{n-1}\times(-1,0))} \approx \|\di( v_j)\|_{L^2(U_j\cap \Omega)}.
        \end{align*}
        For this, see Lemma 2.1.7 in \cite{boffi2013mixed}.
        \item 
        Symmetrically extend the Piola transforms to $(-1,1)^n $ by setting 
        $$\hat{w}_j(x) \coloneq  \left\{
            \begin{array}{ll}
            \hat{v}_j(x), & x_1 <0,\\
            \hat{v}_j(-x_1,x_2,\dots,x_n), & x_1>0.  \\
            \end{array} \right.$$
            
        \item Choose $w_0 = v_0$ and construct the inverse Piola transforms of the extensions $\hat{w}_j$, namely$$~~~w_j  \coloneq  \frac{1}{\det (D \Phi_j )}(D \Phi_j )\cdot \hat{w}_j \circ \Phi_j.$$ 
        For $j=1,...,N,$ the functions $w_j$ are then extensions of $v_j$ to $H^\di(U_j)$ with \begin{align*}
            &\|w_j\|_{L^2((-1,1)^{n-1}\times(-1,0))} \approx \|v_j\|_{L^2(U_j\cap \Omega)} \tn{ and} \\&\|\di (w_j)\|_{L^2((-1,1)^{n-1}\times(-1,0))} \approx \|\di (v_j)\|_{L^2(U_j\cap \Omega)}.
        \end{align*}
        Moreover, each $w_j$ has support compactly contained in $U_j.$
        \item Finally, the sought after $L^2$-bounded extension of $v$ is given by 
        \begin{align*}w(x) \coloneq  &\left\{
            \begin{array}{ll}
            \sum_{j=0}^N w_j(x), &  x \in \bigcup_{j=0}^N U_j\\
            0, & \tn{else}. \\
            \end{array} \right. \qedhere \end{align*}
        \end{enumerate}
    \end{proof}

\sectionnotoc{Acknowledgement}
This work was funded by the Deutsche Forschungsgemeinschaft (DFG, German Research Foundation) – Project-ID 531152215 – CRC 1701.

The author thanks Jochen Glück and  Birgit Jacob for their helpful comments and discussions.\\

\sectionnotoc{Declaration of generative AI and AI-assisted technologies in the manuscript preparation process.}
During the preparation of this work the author used Perplexity, ChatGPT, and Google Gemini for literature search and discussion. No new results of this paper were found, developed or written down using generative AI.  \\

\bibliographystyle{plainurl}
\bibliography{literature}

\end{document}